\documentclass[11pt, draft, english]{article}

\usepackage[utf8]{inputenc}    
\usepackage[T1]{fontenc}       

\usepackage{amsmath}           
\usepackage{amssymb}          
\usepackage[amsmath,thmmarks,hyperref]{ntheorem} 

\usepackage{setspace}          
\usepackage{bbm}               
\usepackage{mathtools}         
\usepackage{mathrsfs}          
\usepackage[multiuser]{fixme}  
\usepackage[a4paper]{geometry} 
\usepackage[shortcuts]{extdash}
\usepackage{enumitem}
\usepackage{upgreek}
\usepackage{tikz-cd}
\usepackage{xcolor}
\usepackage{euflag}
\usepackage[final=true,        
           pdfpagelabels       
           ]{hyperref}         
\hypersetup{
    colorlinks,
    linkcolor={blue!50!black},
    citecolor={blue!50!black},
    urlcolor={blue!50!black}
}

\allowdisplaybreaks

\newcommand{\NN}{\mathbb{N}}
\newcommand{\RR}{\mathbb{R}}
\newcommand{\CC}{\mathbb{C}}

\newcommand{\KK}{\mathbbm{K}}

\newcommand{\algebra}[1]{\mathcal{#1}}
\newcommand{\A}{\algebra{A}}
\newcommand{\B}{\algebra{B}}
\newcommand{\C}{\algebra{C}}
\newcommand{\F}{\algebra{F}}
\renewcommand{\P}{\algebra{P}}
\newcommand{\lie}[1]{\mathfrak{#1}}
\newcommand{\Heisenberg}[1]{\lie g_{\textup{Hei},#1}}

\newcommand{\hermitian}{\mathrm h}
\newcommand{\mloc}{\mathrm{\textup{-}loc}}
\newcommand{\totloc}{{\Zentrum\mloc}}
\newcommand{\cat}{\mathrm{cat}}
\newcommand{\fin}{\mathrm{fin}}

\DeclareMathOperator{\ad}{ad}
\DeclareMathOperator{\RE}{Re}
\DeclareMathOperator{\IM}{Im}
\DeclareMathOperator{\GammaFun}{\Gamma} 
\DeclareMathOperator{\pr}{pr}
\DeclareMathOperator{\id}{id}

\newcommand{\Dom}{\mathcal{D}}
\newcommand{\Eom}{\mathcal{E}}
\newcommand{\Hilb}{\mathcal H}

\newcommand{\Schwartz}{\mathscr{S}}
\newcommand{\Weyl}[1]{\mathcal W_{#1}}
\newcommand{\Adbar}{\mathcal{L}^*}
\newcommand{\Univ}{\mathcal{U}}

\newcommand{\Zentrum}{\mathscr{Z}}
\newcommand{\WeylNil}{\mathcal N}
\newcommand{\dom}{\mathrm{dom}}
\newcommand{\domAna}{\mathrm{dom}^{\upomega}}
\newcommand{\starchar}{{\mathrm{char}^*}}

\DeclarePairedDelimiter{\ordinarySet}{\{}{\}}
\DeclarePairedDelimiter{\ordinaryIP}{\langle}{\rangle}
\DeclarePairedDelimiter{\abs}{|}{|}
\DeclarePairedDelimiter{\norm}{\|}{\|}
\newcommand{\set}[3][]{\ordinarySet[#1]{\,#2 \;#1|\; #3\,}}
\newcommand{\skal}[3][]{\ordinaryIP[#1]{\,#2 \,#1|\, #3\,}}
\newcommand{\kom}[3][]{\mathop{#1[}#2\mathbin{,}#3\mathop{#1]}}
\newcommand{\gen}[2][]{#1\langle\!#1\langle\,#2\,#1\rangle\!#1\rangle}
\newcommand{\genSAlg}[2][]{\gen[#1]{#2}_{*\textup{-}\mathrm{alg}}}

\newcommand{\I}{\mathrm{i}}
\newcommand{\D}{\mathrm{d}}
\newcommand{\argument}{\,\cdot\,}
\newcommand{\cc}[1]{\overline{#1}}
\newcommand{\at}[2][]{#1|_{#2}}
\newcommand{\SchroedNilGen}{\mathcal{B}}
\newcommand{\Ell}{\mathcal{L}}
\newcommand{\iotaUniv}[1]{\operatorname{ \iota^{\scriptscriptstyle\textup{can}}_{{#1\!}}}}
\newcommand{\iotaUnivC}[1]{\operatorname{ \iota^{\scriptscriptstyle\textup{can}}_{{#1,\CC}}}}
\newcommand{\iotaHeisenberg}{\operatorname{ \iota_{\textup{Hei}}}}

\newcommand{\theoremMainInternalTag}{$\clubsuit$}

\theoremheaderfont{\normalfont\bfseries}
\theorembodyfont{\itshape}
\newtheorem{lemma}{Lemma}
\newtheorem{proposition}[lemma]{Proposition}
\newtheorem{theorem}[lemma]{Theorem}
\newtheorem{corollary}[lemma]{Corollary}
\newtheorem{definition}[lemma]{Definition}
\newtheorem*{corollaryNN}{Corollary}

\theoremheaderfont{\normalfont\bfseries}
\theorembodyfont{\normalfont}
\newtheorem{example}[lemma]{Example}
\newtheorem{remark}[lemma]{Remark}

\theorembodyfont{\normalfont}
\makeatletter
\def\thmhead@plain#1#2#3{%
	\thmname{#1}\thmnumber{\@ifnotempty{#1}{ }\@upn{#2}}%
	\thmnote{ {\the\thm@notefont#3}}}
\let\thmhead\thmhead@plain
\makeatother

\theoremheaderfont{\scshape}
\theorembodyfont{\normalfont}
\theoremstyle{nonumberplain}
\theoremseparator{:}
\theoremsymbol{\hbox{$\square$}} 
\newtheorem{proof}{Proof}

\author{
    \textbf{Philipp Schmitt}\thanks{schmitt@math.uni-hannover.de}
    \\
    \and
    \textbf{Matthias Schötz}\thanks{matthias.schoetz@uni-wuerzburg.de}
    \\
}

\title{Real Nullstellensatz for nilpotent Lie algebras}

\begin{document}
	\maketitle
	
	\begin{abstract}
    We prove a real Nullstellensatz for the universal enveloping algebra $\Univ(\lie g)$
    of a nilpotent finite-dimensional real Lie algebra $\lie g$: An ideal of $\Univ(\lie g)$ is real if and only if
    it is the intersection of kernels of irreducible integrable $^*$\=/representations of $\Univ(\lie g)$ by
    (in general unbounded) operators on a Hilbert space. For this, we build upon some classical results about
    the complexification of $\Univ(\lie g)$
    that lead to a complex Nullstellensatz (every semiprime ideal is the intersection of primitive ideals),
    but focus only on the ``real points'', i.e., the $^*$\=/representations on Hilbert spaces. On the
    technical side, this requires to take into account an involution of $\Univ(\lie g)$, namely, the
    antipode of its Hopf algebra structure, which goes into the definition of real ideals of $\Univ(\lie g)$.
    Our result is the first noncommutative real Nullstellensatz for $^*$\=/algebras with unbounded irreducible
    $^*$\=/representations, demonstrating how methods of real algebraic geometry can lead to results about operator
    $^*$\=/algebras also in situations where $C^*$\=/algebraic methods fail.
	\end{abstract}

	{  \small

  }

	\tableofcontents

	\let\thefootnote\relax\footnotetext{Progress on this project has been made over several years at several locations.
	The first author acknowledges support from the EU Staff Exchange project 101086394 ``Operator Algebras That One Can See'',
	partially supported by the University of Warsaw Thematic Research Programme ``Quantum Symmetries''.
	The second author acknowledges support from the project COMPUTE, which is funded within the QuantERA II Programme that
	has received funding from the EU’s H2020 research and innovation programme under the GA No.~101017733 \euflag,
	and support by the Fondation de l’\'Ecole polytechnique as part of the Gaspard Monge Visiting Professor Program.
	Both authors would like to thank the organizers of the ``Math in the Mill'' workshop series in Sondheim vor der Rhön.}

\begin{onehalfspace}

\section{Introduction}
The natural numbers are denoted by $\NN \coloneqq \{1,2,3,\dots\}$ and $\NN_0 \coloneqq \NN \cup \{0\}$,
and the fields of real and complex numbers are denoted by $\RR$ and $\CC$, respectively.

A \emph{$^*$\=/algebra} is a unital associative algebra $\A$ over the field $\CC$ together with an
antilinear involution $\argument^* \colon \A \to \A$ such that $(ab)^* = b^* a^*$ for all $a,b\in \A$.
Its real linear subspace of \emph{hermitian} elements is $\A_\hermitian \coloneqq \set{a\in \A}{a^*=a}$.
A \emph{$^*$\=/subalgebra} of $\A$ is a complex linear subspace of $\A$ that is closed under multiplication,
stable under the $^*$\=/involution,
and contains the unit $1$ of $\A$; hence every $^*$\=/subalgebra is again a $^*$\=/algebra.
The $^*$\=/algebra $\A$ is called
\emph{formally real} if it has the following property:
whenever $a_1,\dots,a_n \in \A$ with $n\in \NN$ fulfill $\sum_{j=1}^n a_j^*a_j = 0$, then $a_1 = \dots = a_n = 0$.
Clearly, every $^*$\=/subalgebra of a formally real $^*$\=/algebra is again formally real.

For example, given any complex pre-Hilbert space $\Dom$ with inner product
$\skal{\argument}{\argument} \colon \Dom \times \Dom \to \CC$ (antilinear in the first, and linear in the second argument),
then the set $\Adbar(\Dom)$ of (linear) endomorphisms $a \colon \Dom \to \Dom$ that admit an
(algebraically) adjoint endomorphism $a^* \colon \Dom \to \Dom$,
i.e., $\skal{a^*\,\phi}{\psi} = \skal{\phi}{a\,\psi}$ for all $\phi,\psi \in \Dom$, is a formally real $^*$\=/algebra.
The unit of $\Adbar(\Dom)$ is the identity map $\id_\Dom \colon \Dom \to \Dom$.
The $^*$\=/subalgebras of $\Adbar(\Dom)$ are called the \emph{$O^*$\=/algebras over $\Dom$} (or operator $^*$\=/algebras).

A \emph{$^*$\=/algebra morphism} $\Phi \colon \A \to \B$ between two $^*$\=/algebras $\A$ and $\B$ is a linear map that fulfills
$\Phi(1) = 1$, $\Phi(aa') = \Phi(a)\,\Phi(a')$, and $\Phi(a^*) = \Phi(a)^*$ for all $a,a'\in \A$.
A \emph{$^*$\=/representation} of a $^*$\=/algebra $\A$ is a tuple $(\Dom,\pi)$ of a complex pre-Hilbert space $\Dom$ and a
$^*$\=/algebra morphism $\pi \colon \A \to \Adbar(\Dom)$, and we say that this $^*$\=/representation $(\Dom,\pi)$ is
\emph{faithful} if $\pi$ is injective.
So every $^*$\=/algebra that admits a faithful $^*$\=/representation is isomorphic to some $O^*$\=/algebra,
and therefore it must be formally real.

It is a mostly open problem to understand when the converse is also true.
Only few positive results are known. Namely, consider a formally real $^*$\=/algebra $\A$,
then there exists a faithful $^*$\=/representation of $\A$ in each of the following cases:

\begin{itemize}
  \item if $\A$ is a finitely generated commutative $^*$\=/algebra
    (this is the classical real Nullstellensatz for the polynomial algebra by Dubois \cite{dubois:nullstellensatzForOrderedFields} and Risler \cite{risler:nullstellensatz});
  \item if $\A$ is a finite-dimensional $^*$\=/algebra
    (this follows e.g.~from the Wedderburn--Artin structure theorems for semi-simple artinian rings,
    see \cite[Thm.~9.7.22]{palmer:banachAlgebrasAndTheGeneralTheoryOfStarAlgebras} for details);
  \item if $\A$ is a quotient of a $^*$\=/algebra of polynomials with matrix coefficients, \cite{cimpric:realAlgebraicGeometryForMatricesOverCommutativeRings}.
\end{itemize}
In all these cases, a faithful $^*$\=/representation can be obtained as a (possibly infinite)
direct sum of irreducible finite-dimensional $^*$\=/representations.
To the best of our knowledge, there are no positive general results known for situations where $^*$\=/representations on infinite-dimensional
spaces are required, let alone unbounded $^*$\=/representations, with the exception of our predecessor article \cite{schmitt.schoetz:realNullstellensatzForTwoStepNilpotentLieAlgebras}
treating the $2$-step nilpotent case. On the other hand, it is known that there are finitely generated
formally real $^*$\=/algebras that do not admit a faithful $^*$\=/representation, e.g., the $^*$\=/algebra generated by two elements
$a,x$ and their formal adjoints $a^*,x^*$, subject to the relations $xx^* + aa^* = 1$ and $a^*a = q aa^*$ for some
fixed number $q \in {]0,1[}$; see \cite[Prop.~1.8]{cimpricEtAl:realNullstellensatzGegenbeispiel}.

There are of course other results characterizing the existence of faithful $^*$\=/representations of certain types of $^*$\=/algebras
by other properties. E.g., the celebrated Gelfand--Naimark theorem guarantees that any $^*$\=/algebra that admits a $C^*$\=/norm has a
faithful $^*$\=/representation (by bounded operators), and \cite{schoetz:gelfandNaimarkTheorem} generalizes this to $\sigma$\=/bounded
$^*$\=/algebras that admit a certain closed partial order on the hermitian elements (but now unbounded $^*$\=/representations might be required).
Alternatively, \cite{popovych:onOStarRepresentabilityAndCStarRepresentabilityOfStarAlgebras}
gives a sufficient condition for the existence of a faithful $^*$\=/representation of a $^*$\=/algebra $\A$ via properties of a Gröbner basis of $\A$.
However, none of these provide a characterization by a simple algebraic property like being formally real.

Any $^*$\=/algebra $\A$ becomes a (complex) Lie algebra in a natural way, namely by endowing the vector space underlying $\A$
with the \emph{commutator bracket} $\kom{\argument}{\argument} \colon \A \times \A \to \A$,
\begin{equation}
  \label{eq:commutatorDef}
  (a,b) \mapsto \kom{a}{b} \coloneqq ab-ba.
\end{equation}
Note that $\kom{a}{b}^* = \kom{b^*}{a^*}$ for all $a,b\in \A$. By a \emph{Lie $^*$\=/subalgebra} of $\A$ we mean a
(complex) linear subspace $\lie g$ of $\A$ such that $a^* \in \lie g$ and $\kom{a}{b} \in \lie g$ for all $a,b\in \lie g$.

Recall that a finite-dimensional Lie algebra $\lie g$ is called \emph{nilpotent} if there exists $n\in \NN_0$ such that
$\ad_{a_1} \circ \cdots \circ \ad_{a_n} = 0$ for all $a_1,\dots,a_n \in \lie g$, where $\ad_a \colon \lie g \to \lie g$ is defined as
$\ad_ab \coloneqq \kom{a}{b}$ for all $a,b\in \lie g$. Our main motivation is to obtain the following:

\begin{corollaryNN}
  Let $\A$ be a $^*$\=/algebra that is generated by a finite-dimensional Lie $^*$\=/subalgebra $\lie g$ of $\A$.
  Assume $\lie g$ is nilpotent. Then the following are equivalent:
  \begin{enumerate}[ref=\textit{\roman*.)},label=\textit{\roman*.)}]
    \item $\A$ is formally real.
    \item There exists a faithful $^*$\=/representation of $\A$.
    \item \label{item:integrableFaithfulRep} There exists a faithful integrable $^*$\=/representation of $\A$.
  \end{enumerate}
\end{corollaryNN}
In part~\ref{item:integrableFaithfulRep}, \emph{integrability} of a $^*$\=/representation $(\Dom,\pi)$
essentially means that $(\Dom,\pi)$ is the derivative of a strongly continuous unitary
representation of a (real) Lie group $G$ whose complexified Lie algebra is $\lie g$.
Such an additional regularity condition is desirable
because the $^*$\=/representations in question are typically not bounded, and unbounded $^*$\=/representations
in general may exhibit rather pathological behaviour. Details are discussed in Section~\ref{sec:integrability}.

We will obtain the above Corollary from our main Theorem~\ref{theorem:main},
which generalizes the real Nullstellensatz of real algebraic geometry from the real polynomial algebra to the universal
enveloping $^*$\=/algebra of a nilpotent finite-dimensional Lie $^*$\=/algebra (a complex Lie algebra with an antilinear
involution that provides a real form): every real ideal is the intersection of kernels of irreducible integrable
$^*$\=/representations.
This is a ``real'' analog of a well-known ``complex'' Nullstellensatz for the universal enveloping algebra of
finite-dimensional complex Lie algebras:
every semiprime ideal is the intersection of primitive ideals (see e.g.~\cite[Prop.~3.1.15]{dixmier:envelopingAlgebras}).
This purely algebraic result, however, does not allow one to draw the connection to $^*$\=/representations on Hilbert spaces.

The proof of  Theorem~\ref{theorem:main} combines methods of real algebraic geometry and operator theory.
Applications of functional analysis and operator theory to commutative and noncommutative real algebraic geometry
have been extensively studied over the last decades (most notably the various incarnations of the archimedean
Positivstellensatz \cite{krivine:positivstellensatz, jacobi:representationTheoremForPartiallyOrderedRings, helton.mcCullough:positivstellensatzForNonCommutativePolynomials}
and its applications to the moment problem, optimization, and quantum information theory,
e.g.~\cite{schmuedgen:KMomentProblemForCompactSemiAlgebraicSets, lasserre:globalOptimizationWithPolynomials, navascues.pironio.acin:convergentHierarchyOfSemidefiniteProgramsCharacterizingTheSetOfQuantumCorrelations}).
Conversely, Theorem~\ref{theorem:main} applies techniques of real algebraic geometry
to the representation theory of unbounded $^*$\=/algebras.

Roughly speaking, the toolbox of real algebraic geometry contains two mostly separate compartments: Some methods exploit
the fact that $\RR$ is the (unique) complete archimedean ordered field, while others rely on $\RR$ being (one of many)
real closed fields. The first type of techniques (complete archimedean ordered field) leads to e.g.~various strict Positivstellensätze,
characterizing the cone of in some sense ``strictly'' positive elements of $^*$\=/algebras. Examples are the archimedean
Positivstellensatz mentioned above and its unbounded
generalizations~\cite{marshall:extendingArchimedeanPositivstellensatzToTheNonCOmpactCase, schmuedgen:StrictPositivstellensatzForWeylAlgebra, schmuedgen:StrictPositivstellensatzForEnvelopingAlgebras}.
Typically, these results follow from some convex separation argument that provides a positive linear functional to which
the GNS construction can be applied. In rare cases, in particular when
one can prove closedness of some cone of positive elements by elementary means, one can also obtain non-strict Positivstellensätze this way,
e.g.~\cite{helton:positiveNoncommutativePolynomialsAreSumsOfSquares, helton.mcCullough.putinar:nonCommutativePositivstellensatzOnIsometries, schmitt.schoetz:symmetryReductionOfStatesII}.
However, other non-strict Positivstellensätze, like the one by Krivine and Stengle for polynomials on semi\-algebraic sets
\cite{krivine:positivstellensatz, stengle:positivstellensatz},
seem to require the second type of techniques (real closed fields). The same is true for the related real Nullstellensatz
for the polynomial algebra \cite{dubois:nullstellensatzForOrderedFields, risler:nullstellensatz}.
These theorems can be proven by first constructing evaluations in some real closed field extension of $\RR$,
and then applying the Tarski--Seidenberg principle in order to obtain real points.
A notable hybrid that combines both types of techniques is Schmüdgen's Positivstellensatz \cite{schmuedgen:KMomentProblemForCompactSemiAlgebraicSets}
for polynomials on compact semialgebraic sets:
its proof first applies the Krivine--Stengle Positivstellensatz in order to obtain an algebraic certificate of boundedness from the compactness assumption,
which then allows one to apply the archimedean Positivstellensatz.

Our proof of Theorem~\ref{theorem:main} is another hybrid: we combine the real Nullstellensatz for the polynomial algebra
(applied to suitable commutative real subalgebras of in general noncommutative $^*$\=/algebras)
with some classical structure results for enveloping algebras of nilpotent Lie algebras
(see e.g.~Dixmier's monograph~\cite{dixmier:envelopingAlgebras}) that provide the required $^*$\=/representations,
and with explicit analyticity estimates that prove integrability of these $^*$\=/representations.

This article is structured as follows: In the next Section~\ref{sec:integrability} we present our main
Theorem~\ref{theorem:main} after discussing some details concerning
enveloping $^*$\=/algebras of Lie $^*$\=/algebras and their integrable $^*$\=/representations.
A particular example for these notions is the Weyl $^*$\=/algebra and its Schrödinger $^*$\=/representation, which are examined in detail in
Section~\ref{sec:weyl}. In Section~\ref{sec:centralEnv} we adapt some classical structure results for enveloping algebras of nilpotent
Lie algebras to the $^*$\=/algebraic setting, and in Section~\ref{sec:realprime} we examine real prime ideals.
In particular, in Theorem~\ref{theorem:realIdealDecomposition} we show that every real ideal of a Noetherian $^*$\=/algebra
is the intersection of finitely many real prime ideals, which might be of interest also for other applications.
By combining these results we prove Theorem~\ref{theorem:main} in Section~\ref{sec:proof}.
Finally, in Section \ref{sec:realuniv} we discuss the related case of the (real) enveloping algebra of a nilpotent finite-dimensional real Lie algebra,
which is an involutive real Hopf algebra.
Theorem~\ref{theorem:main} can easily be adapted to this setting, resulting in Theorem~\ref{theorem:realUniv},
but the converse procedure would fail, i.e., from Theorem~\ref{theorem:realUniv} alone one could not deduce Theorem~\ref{theorem:main}
in a similarly straightforward way. We see this as an indication that $^*$\=/algebras (not involutive real algebras)
should be the chief object of study in noncommutative real algebraic geometry.

\section{Integrable \texorpdfstring{$^*$-representations}{*-representations} and the real Nullstellensatz} \label{sec:integrability}
By a \emph{Lie $^*$\=/algebra} we mean a complex Lie algebra $\lie g$ endowed with an antilinear involution
$\argument^* \colon \lie g \to \lie g$ such that $\kom{x}{y}^* = \kom{y^*}{x^*}$ for all $x,y \in \lie g$.
Its real linear subspace of \emph{hermitian} elements is $\lie g_\hermitian \coloneqq \set{x\in \lie g}{x^*=x}$.
In particular, any $^*$\=/algebra $\A$ with the commutator bracket from \eqref{eq:commutatorDef} is a Lie $^*$\=/algebra.
\emph{Lie $^*$\=/subalgebras} are Lie subalgebras that are stable under the $^*$\=/involution.
A \emph{Lie $^*$\=/algebra morphism} $\phi \colon \lie g \to \lie h$ between two Lie $^*$\=/algebras
is a linear map that fulfills $\phi\bigl(\kom{x}{x'}\bigr) = \kom[\big]{\phi(x)}{\phi(x')}$ and
$\phi(x^*) = \phi(x)^*$ for all $x,x' \in \lie g$.
Note that Lie $^*$\=/algebras are equivalent to real Lie algebras: For every Lie $^*$\=/algebra
$\lie g$, its \emph{antihermitian part} $\I \lie g_\hermitian = \set{\I x}{x\in \lie g_\hermitian}$ with the
restricted Lie bracket is a real Lie algebra; conversely, for every real Lie algebra $\lie h$, its complexification
$\lie h \oplus \I \lie h$ with the $\CC$-linearly extended Lie bracket and the $^*$\=/involution
$(y+\I y')^* \coloneqq -y + \I y'$ for all $y,y'\in \lie h$ is a Lie $^*$\=/algebra.
An \emph{enveloping $^*$\=/algebra} of a Lie $^*$\=/algebra $\lie g$ is a tuple
$(\A, \iota)$ of a $^*$\=/algebra $\A$ and a Lie $^*$\=/algebra morphism $\iota \colon \lie g\to \A$
(not necessarily injective)
such that $\A$ is generated (as a $^*$\=/algebra) by $\iota(\lie g)$, i.e., every element of $\A$
can be expressed as a linear combination of arbitrary products (including the empty product $1$) of elements of $\iota(\lie g)$.

\begin{remark} \label{remark:envelopingquotient}
  Recall that for any Lie algebra $\lie g$ over a field $\KK$, the \emph{universal enveloping algebra} $\Univ(\lie g)$ is
  an (associative) $\KK$\=/algebra $\Univ(\lie g)$ together with a Lie algebra morphism $\iotaUniv{\lie g} \colon \lie g \to \Univ(\lie g)$
  such that for every other tuple of a $\KK$\=/algebra $\A$ and a Lie algebra morphism $\phi \colon \lie g \to \A$
  there exists a unique algebra morphism $\Phi \colon \Univ(\lie g) \to \A$ that extends $\phi$ in the sense that
  $\Phi \circ \iotaUniv{\lie g} = \phi$. If $\lie g$ is a Lie $^*$\=/algebra (and in particular, $\KK = \CC$), then
  one can equip $\Univ(\lie g)$ with an appropriate $^*$\=/involution such that $\Univ(\lie g)$ becomes a $^*$\=/algebra,
  $\iotaUniv{\lie g}$ a Lie $^*$\=/algebra morphism, and then $\Phi$ is a $^*$\=/algebra morphism if and only if $\phi$ is a
  Lie $^*$\=/algebra morphism. Therefore, every enveloping $^*$\=/algebra of $\lie g$ can be realised as a quotient
  of $\Univ(\lie g)$.
  For finite-dimensional $\lie g$, $\Univ(\lie g)$ is Noetherian by~\cite[Cor.~7.4]{mcconnell.robson:noncommutativeNoetherianRings},
  hence all its quotients are also Noetherian.
\end{remark}

Consider a $^*$\=/algebra $\A$ and a $^*$\=/representation $(\Dom,\pi)$ of $\A$. A linear subspace
$\Eom$ of $\Dom$ is called \emph{$\pi$\=/invariant} if $\pi(a)\,\psi \in \Eom$ for all $a\in \A$, $\psi \in \Eom$.
In this case, the $^*$\=/representation $(\Dom,\pi)$ can be restricted to a $^*$\=/representation $(\Eom,\pi\at{\Eom})$,
where $\pi\at{\Eom}(a)\,\psi \coloneqq \pi(a)\,\psi$ for all $a\in \A$, $\psi \in \Eom$.
Write $\Hilb$ for the completion of $\Dom$ to a Hilbert space and $\norm{\argument}$ for the Hilbert norm on $\Hilb$,
then we can treat $\Dom$ as a
$\norm{\argument}$\=/dense linear subspace of $\Hilb$. For all $a\in \A$, the linear map
$\pi(a) \colon \Dom \to \Dom$ is a linear operator on $\Hilb$ (in general unbounded).
The domain of the adjoint operator of $\pi(a)$ is
\begin{equation}
  \label{eq:daggerDef}
  \dom\bigl(\pi(a)^\dagger\big)
  \coloneqq
  \set[\Bigg]{\phi \in \Hilb}{
    \begin{array}{l}
      \text{there is a vector $\pi(a)^\dagger\phi \in \Hilb$ such that}\\
      \text{$\skal{\phi}{\pi(a)\,\psi} = \skal{\pi(a)^\dagger\phi}{\psi}$ for all $\psi \in \Dom$}
    \end{array}
  }
\end{equation}
and $\Dom \subseteq \dom\bigl(\pi(a)^\dagger\big)$ because $\skal{\phi}{\pi(a)\,\psi} = \skal{\pi(a^*)\,\phi}{\psi}$
for all $\phi,\psi \in \Dom$. We write
\begin{equation}
  \dom(\pi^\dagger) \coloneqq \bigcap_{a\in\A} \dom\bigl(\pi(a)^\dagger\big),
\end{equation}
and we call the $^*$\=/representation $(\Dom,\pi)$ \emph{selfadjoint} if $\Dom = \dom(\pi^\dagger)$.
Such a selfadjoint $^*$\=/representation $(\Dom,\pi)$ is \emph{irreducible} if $\Dom \neq \{0\}$ and its \emph{bounded commutant}
\begin{equation}
  \pi(\A)'
  \coloneqq
  \set[\big]{
    \!C \colon \Dom \to \Dom
  }{
    \text{$C$ is a $\norm{\argument}$-bounded linear map such that}~C \,\pi(a) = \pi(a)\, C~\text{for all $a\in\A$}\!
  }\!\!
\end{equation}
is trivial, i.e., $\pi(\A)' = \set{\lambda\,\id_\Dom}{\lambda \in \CC}$.
Selfadjointness of a $^*$\=/representation guarantees that various reasonable notions of irreducibility are equivalent,
see \cite[Sec.~7.2 and Remark~3 in Sec.~8.3]{schmuedgen:UnboundedOperatorAlgebraAndRepresentationTheory}.
In particular, if a selfadjoint $^*$\=/representation $(\Dom,\pi)$ of a $^*$\=/algebra $\A$ is irreducible
and $\Eom$ is a $\pi$\=/invariant linear subspace of $\Dom$ such that the restricted $^*$\=/representation
$(\Eom,\pi\at\Eom)$ is selfadjoint, then $\Eom = \{0\}$ or $\Eom = \Dom$,
see \cite[Cor.~8.3.12]{schmuedgen:UnboundedOperatorAlgebraAndRepresentationTheory}.
Nevertheless, selfadjoint $^*$\=/representations may still exhibit rather pathological
behaviour, even for finitely generated commutative $^*$\=/algebras $\A$,
see e.g.~\cite[Sec.~9.4]{schmuedgen:UnboundedOperatorAlgebraAndRepresentationTheory}.

Now let $(\A,\iota)$ be an enveloping $^*$\=/algebra of a finite-dimensional Lie $^*$\=/algebra $\lie g$. In this case
there is a well established notion of wellbehaved $^*$\=/representations of $\A$, namely \emph{integrable} $^*$\=/representations.
Essentially, a $^*$\=/representation $(\Dom,\pi)$ of $(\A,\iota)$ is integrable if it arises from differentiating a strongly
continuous unitary representation of a real Lie group $G$ that integrates the real Lie algebra $\I \lie g_\hermitian$.
We will use an equivalent characterization of integrable $^*$\=/representations due to Nelson \cite{nelson:analyticVectors}:
Fix any basis $B$ of the linear subspace $\pi(\iota(\lie g))$ of $\Adbar(\Dom)$.
The space of \emph{analytic vectors} for the $^*$\=/representation $(\Dom,\pi)$ of $(\A,\iota)$ is defined as
\begin{equation}
  \domAna(\pi,\iota)
  \coloneqq
  \set[\Bigg]{
    \psi \in \Dom
  }{
    \begin{array}{l}
      \text{there is $\rho \in {[0,\infty[}$ such that for all $\ell \in \NN$ and $H_1,\dots,H_\ell \in B$}\\
      \text{the estimate $\norm{H_\ell \cdots H_1\,\psi} \le \rho^\ell\, \ell!$ holds}
    \end{array}
  }.
\end{equation}
It is straightforward to check that $\domAna(\pi,\iota)$ depends only on the linear subspace $\pi(\iota(\lie g))$ of $\Adbar(\Dom)$,
but not on the specific choice of the basis $B$.
Moreover, $\domAna(\pi,\iota)$ is a $\pi$\=/invariant linear subspace of $\Dom$.
The $^*$\=/representation $(\Dom,\pi)$ of the enveloping $^*$\=/algebra $(\A,\iota)$ of $\lie g$ is integrable
if and only if $\domAna(\pi,\iota)$ is $\norm{\argument}$\=/dense in $\Dom$
and $(\Dom,\pi)$ selfadjoint, see \cite[Thm.~10.5.8]{schmuedgen:UnboundedOperatorAlgebraAndRepresentationTheory}.
If only the first of these two conditions is fulfilled, it still gives rise to a (possibly larger) integrable $^*$\=/representation:

\begin{proposition} \label{proposition:exponentiation}
  Let $(\A,\iota)$ be an enveloping $^*$\=/algebra of a finite-dimensional Lie $^*$\=/algebra $\lie g$
  and $(\Dom,\pi)$ a $^*$\=/representation of $\A$.
  Assume that $\domAna(\pi,\iota)$ is $\norm{\argument}$\=/dense in $\Dom$.
  Then for all $a\in \A$ and $\phi \in \dom(\pi^\dagger)$ there is a unique element $\pi^\dagger(a)\,\phi \in \dom(\pi^\dagger)$
  such that
  \begin{equation}
    \label{eq:adjointRepDef}
    \skal{\phi}{\pi(a^*)\,\psi} = \skal{\pi^\dagger(a)\,\phi}{\psi}\quad\quad\text{for all }\psi \in \Dom
    ,
  \end{equation}
  and this assignment $\dom(\pi^\dagger) \ni \phi \mapsto \pi^\dagger(a)\,\phi \in \dom(\pi^\dagger)$ defines an element
  $\pi^\dagger(a) \in \Adbar\bigl( \dom(\pi^\dagger) \bigr)$.
  Moreover, the resulting map $\pi^\dagger \colon \A \to \Adbar\bigl( \dom(\pi^\dagger) \bigr)$ is a $^*$\=/algebra
  morphism, and $\bigl( \dom(\pi^\dagger), \pi^\dagger \bigr)$ is an integrable $^*$\=/representation of $(\A,\iota)$.
\end{proposition}
\begin{proof}
  Write $\Hilb$ for the completion of $\Dom$ to a Hilbert space.
  Let $a\in \A$ and $\phi \in \dom(\pi^\dagger)$.
  By definition of $\dom(\pi^\dagger)$, there exists an element
  $\pi^\dagger(a)\,\phi \in \Hilb$ such that \eqref{eq:adjointRepDef} holds (in the notation of
  \eqref{eq:daggerDef}, $\pi^\dagger(a)\,\phi = \pi(a^*)^\dagger\,\phi$), and $\pi^\dagger(a)\,\phi$ is uniquely
  determined by \eqref{eq:adjointRepDef} because $\Dom$ is dense in $\Hilb$.
  This is the so-called adjoint representation of $(\Dom,\pi)$,
  see \cite[Def.~8.1.4]{schmuedgen:UnboundedOperatorAlgebraAndRepresentationTheory}.
  In particular, $\pi^\dagger(a)\,\phi \in \dom(\pi^\dagger)$ for all $a\in \A$, $\phi \in \dom(\pi^\dagger)$
  by \cite[Prop.~8.1.3]{schmuedgen:UnboundedOperatorAlgebraAndRepresentationTheory}.
  The assumption that $\domAna(\pi,\iota)$ is $\norm{\argument}$\=/dense in $\Dom$ guarantees that $(\Dom,\pi)$ is an
  exponentiable $^*$\=/representation of $(\A,\iota)$ as in
  \cite[Cor.~10.5.5]{schmuedgen:UnboundedOperatorAlgebraAndRepresentationTheory}.
  Therefore, by \cite[Prop.~10.5.2]{schmuedgen:UnboundedOperatorAlgebraAndRepresentationTheory},
  $\bigl( \dom(\pi^\dagger), \pi^\dagger \bigr)$ is an integrable $^*$\=/representation of $(\A,\iota)$.
\end{proof}
Note that it is a non-trivial consequence of Proposition~\cite[Prop.~10.5.2]{schmuedgen:UnboundedOperatorAlgebraAndRepresentationTheory}
that $\bigl( \dom(\pi^\dagger), \pi^\dagger \bigr)$ is in fact a $^*$\=/representation
(e.g., see \cite[Ex.~8.1.14]{schmuedgen:UnboundedOperatorAlgebraAndRepresentationTheory}).
We also mention that by \cite[Cor.~10.2.18]{schmuedgen:UnboundedOperatorAlgebraAndRepresentationTheory},
an integrable $^*$\=/representation is irreducible if and only if an integrating unitary group representation
(restricted to the connected component of the unit) is irreducible.

Integrability and irreducibility depend only on properties of the represented Lie $^*$\=/algebra:

\begin{proposition} \label{proposition:irredintstability}
  Let $\Dom$ be a pre-Hilbert space.
  Moreover, for $k\in \{1,2\}$, let $\lie g_k$ be a finite-dimensional Lie $^*$\=/algebra, $(\A_k,\iota_k)$ an enveloping $^*$\=/algebra
  of $\lie g_k$, and $\pi_k \colon \A_k \to \Adbar(\Dom)$ a $^*$\=/algebra morphism.
  If $(\Dom,\pi_1)$ is an irreducible integrable $^*$\=/representation of $(\A_1,\iota_1)$
  and $\pi_1(\iota_1(\lie g_1)) = \pi_2(\iota_2(\lie g_2))$, then
  $(\Dom,\pi_2)$ is an irreducible integrable $^*$\=/representation of $(\A_2,\iota_2)$.
\end{proposition}

\begin{proof}
  $\pi_1(\iota_1(\lie g_1)) = \pi_2(\iota_2(\lie g_2))$ implies $\pi_1(\A_1) = \pi_2(\A_2)$
  and $\domAna(\pi_1,\iota_1) = \domAna(\pi_2,\iota_2)$.
  Therefore, $\set{\lambda \id_\Dom}{\lambda \in \CC} = \pi_1(\A_1)' = \pi_2(\A_2)'$,
  $\Dom = \bigcap_{a\in\A_1} \dom\bigl(\pi_1(a)^\dagger\big) = \bigcap_{a\in\A_2} \dom\bigl(\pi_2(a)^\dagger\big)$,
  and $\domAna(\pi_2,\iota_2)$ is $\norm{\argument}$\=/dense in $\Dom$ because $\domAna(\pi_1,\iota_1)$ is.
\end{proof}

Irreducible integrable $^*$\=/representations
are the noncommutative generalization of ``real points'' (or rather, $^*$\=/characters) in real algebraic geometry:
Let $\A$ be a commutative $^*$\=/algebra generated by finitely many hermitian elements $x_1,\dots,x_d$ with $d\in \NN$.
Define $\iota \colon \CC^d \to \A$ as the linear map that sends the $j$\=/th standard basis vector $e_j$ of $\CC^d$
to $x_j$, then $(\A,\iota)$ is an enveloping $^*$\=/algebra of the abelian Lie $^*$\=/algebra $\CC^d$
(with the $0$-bracket and the $^*$\=/involution given by elementwise complex conjugation).
A \emph{$^*$\=/character} of $\A$ is a $^*$\=/homomorphism $\varphi \colon \A \to \CC$.
As $\CC \cong \Adbar(\Hilb)$ for the $1$\=/dimensional Hilbert space $\Hilb$, a $^*$\=/character $\varphi$ of $\A$
is the same as a $1$\=/dimensional $^*$\=/representation $(\Hilb,\varphi)$ of $\A$,
and it is easy to check that $(\Hilb,\varphi)$ is even an irreducible integrable $^*$\=/representation of $(\A,\iota)$.
By \cite[Cor.~9.1.11]{schmuedgen:UnboundedOperatorAlgebraAndRepresentationTheory},
these $1$\=/dimensional $^*$\=/representations are the only irreducible integrable ones.

This correspondence between $^*$\=/characters and irreducible integrable $^*$\=/representations naturally leads to the question,
whether the classical real Nullstellensatz can be extended to enveloping $^*$\=/algebras, for a suitable notion of real ideals.
Namely, an ideal $I$ of a $^*$\=/algebra $\A$ is called \emph{real} if the following holds:
Whenever $k\in \NN$ and $a_1,\dots,a_k \in \A$ fulfill $\sum_{j=1}^k a_j^* a_j \in I$, then $a_1,\dots,a_k \in I$.
For example, it is easy to check that for every $^*$\=/representation $(\Dom,\pi)$ of $\A$,
the kernel $\ker(\pi)$ is a real ideal of $\A$ because $\Adbar(\Dom)$ is formally real.

\begin{definition} \label{definition:realNullstellensatz}
  Let $\lie g$ be a finite-dimensional Lie $^*$\=/algebra. We say that \emph{the real Nullstellensatz holds for $\lie g$}
  if for every enveloping $^*$\=/algebra $(\A,\iota)$ of $\lie g$ and all ideals $I$ of $\A$ the following are equivalent:
  \begin{enumerate}
    \item \label{item:realNullstellensatz:real} $I$ is real,
    \item \label{item:realNullstellensatz:ker} there is a set $X$ of irreducible integrable $^*$\=/representations of $(\A,\iota)$
      such that
      \begin{equation}
        I = \bigcap_{(\Dom,\pi)\in X} \ker(\pi),
      \end{equation}
      where, by convention, the empty intersection is $\bigcap_{(\Dom,\pi)\in \emptyset} \ker(\pi) \coloneqq \A$.
  \end{enumerate}
\end{definition}
As the kernel of any $^*$\=/representation is a real ideal and as one easily checks that intersections of real ideals are again real,
the implication \ref{item:realNullstellensatz:ker}~$\implies$~\ref{item:realNullstellensatz:real} is always true.
So the real Nullstellensatz holds for $\lie g$ if and only if the converse implication
\ref{item:realNullstellensatz:real}~$\implies$~\ref{item:realNullstellensatz:ker} holds.
Moreover, by Remark~\ref{remark:envelopingquotient}, it would actually be sufficient to consider
only the case $(\A,\iota) = \bigl(\Univ(\lie g),\iotaUniv{\lie g}\bigr)$.

One motivation for Definition~\ref{definition:realNullstellensatz} comes from the observation
that such a noncommutative real Nullstellensatz leads to a representation theorem for the corresponding quotient $^*$\=/algebras:

\begin{proposition} \label{proposition:mainCorollary}
  Let $(\A,\iota)$ be an enveloping $^*$\=/algebra of a finite-dimensional Lie $^*$\=/algebra $\lie g$
  and assume that the real Nullstellensatz holds for $\lie g$. Then the following are equivalent:
  \begin{enumerate}
    \item \label{item:prop:mainCorollary:real} $\A$ is formally real.
    \item \label{item:prop:mainCorollary:faithful} There exists a faithful $^*$\=/representation of $\A$.
    \item \label{item:prop:mainCorollary:integrable} There exists a faithful integrable $^*$\=/representation of $\A$.
  \end{enumerate}
\end{proposition}
\begin{proof}
  The implication \ref{item:prop:mainCorollary:integrable}~$\implies$~\ref{item:prop:mainCorollary:faithful} is trivial,
  and the implication \ref{item:prop:mainCorollary:faithful}~$\implies$~\ref{item:prop:mainCorollary:real} is easy to check
  (which has been discussed already in the introduction). It remains to prove the implication
  \ref{item:prop:mainCorollary:real}~$\implies$~\ref{item:prop:mainCorollary:integrable}, so assume that $\A$ is formally real,
  i.e., the ideal $\{0\}$ of $\A$ is real. As the real Nullstellensatz holds for $\lie g$ by assumption,
  there is a set $X$ of irreducible integrable $^*$\=/representations of $(\A,\iota)$ such that $\bigcap_{(\Dom,\pi)\in X} \ker(\pi) = \{0\}$.
  We define a new pre-Hilbert space as the direct sum
  \begin{equation*}
    \Eom \coloneqq \bigoplus_{(\Dom,\pi) \in X} \Dom,
  \end{equation*}
  and for every $a\in \A$, we define the linear map $\rho(a) \colon \Eom \to \Eom$,
  \begin{equation*}
    \sum_{(\Dom,\pi) \in X} \psi_{(\Dom,\pi)}
    \mapsto
    \rho(a) \sum_{(\Dom,\pi) \in X} \psi_{(\Dom,\pi)}
    \coloneqq
    \sum_{(\Dom,\pi) \in X} \pi(a)\,\psi_{(\Dom,\pi)}
    .
  \end{equation*}
  It is straightforward to check that this way, we obtain a $^*$\=/algebra morphism $\rho \colon \A \to \Adbar(\Eom)$.
  From $\bigcap_{(\Dom,\pi)\in X} \ker(\pi) = \{0\}$ it follows that $\ker(\rho) = \{0\}$.
  So $(\Eom,\rho)$ is a faithful $^*$\=/representation of $\A$.

  For every $(\Dom,\pi) \in X$ we can treat the pre-Hilbert space $\Dom$ as a linear subspace of $\Eom$.
  Under this convention, the identity $\rho(a)\,\psi = \pi(a)\,\psi$ holds for all $a\in \A$, $(\Dom,\pi) \in X$, and $\psi \in \Dom$.
  The spaces of analytic vectors $\domAna(\pi,\iota)$ for any $(\Dom,\pi) \in X$ then are also linear subspaces of $\Eom$,
  and clearly $\domAna(\pi,\iota) \subseteq \domAna(\rho,\iota)$ for all $(\Dom,\pi) \in X$.
  As $\domAna(\rho,\iota)$ is a linear subspace of $\Eom$, it follows that
  $\bigoplus_{(\Dom,\pi) \in X} \domAna(\pi,\iota) \subseteq \domAna(\rho,\iota)$.
  Moreover, $\norm{\argument}$\=/denseness of $\domAna(\pi,\iota)$ in $\Dom$ for all $(\Dom,\pi) \in X$,
  implies that $\bigoplus_{(\Dom,\pi) \in X} \domAna(\pi,\iota)$, and consequently also $\domAna(\rho,\iota)$,
  are $\norm{\argument}$\=/dense in $\Eom$.
  Let $\bigl( \dom(\rho^\dagger), \rho^\dagger \bigr)$ be the integrable $^*$\=/representation
  of $(\A,\iota)$ from Proposition~\ref{proposition:exponentiation}.
  For all $a\in \A$ and $\phi \in \Eom$, the identity $\rho^\dagger(a)\,\phi = \rho(a)\,\phi$ holds,
  because $\skal{\phi}{\rho(a^*)\,\psi} = \skal{\rho(a)\,\phi}{\psi}$ for all $\psi \in \Eom$.
  It follows that $\ker(\rho^\dagger) \subseteq \ker(\rho) = \{0\}$, so
  $\bigl( \dom(\rho^\dagger), \rho^\dagger \bigr)$ is also faithful.
\end{proof}

Unfortunately, little is known about the class of Lie $^*$\=/algebras for which the real Nullstellensatz holds.
In the commutative case we have the following consequence of the classical real Nullstellensatz for the polynomial algebra:

\begin{theorem}[classical real Nullstellensatz] \label{theorem:NullstellensatzCommutative}
  Let $\A$ be a finitely generated and formally real commutative $^*$\=/algebra and write
  $\starchar(\A)$ for the set of $^*$\=/characters of $\A$. Then $\bigcap_{\varphi \in \starchar(\A)} \ker(\varphi) = \{0\}$.
\end{theorem}
\begin{proof}
  As $\A$ is commutative, its hermitian part $\A_\hermitian$ is a commutative real algebra.
  Every $a\in \A$ can be decomposed as $a = a_r + \I a_i$ with uniquely determined $a_r,a_i \in \A_\hermitian$,
  namely $a_r = \frac12(a+a^*)$ and $a_i = \frac{1}{2\I}(a-a^*)$. By decomposing the (finitely many) generators
  of $\A$ in this way, we obtain finitely many generators $g_1,\dots,g_n$ of the commutative real algebra $\A_\hermitian$,
  where $n\in \NN$. Let $\pr \colon \RR[x_1,\dots,x_n] \to \A_\hermitian$
  be the unique algebra homomorphism that maps $x_i$ to $g_i$ for all $i\in \{1,\dots,n\}$.
  Note that $\pr$ is surjective.
  If $k\in \NN$ and $p_1,\dots,p_k \in \RR[x_1,\dots,x_n]$ fulfill
  $\sum_{j=1}^k p_j^2 \in \ker(\pr)$, then $\sum_{j=1}^k \pr(p_j)^2 =0$, so $\pr(p_1) = \dots = \pr(p_k) = 0$
  because $\A$ is formally real, and consequently $p_1,\dots,p_k \in \ker(\pr)$.
  Therefore, the Nullstellensatz for the real polynomial algebra applies and shows that there is a subset $X \subseteq \RR^n$ such that
  $\ker(\pr) = \set{p\in \RR[x_1,\dots,x_n]}{p(\xi) = 0~\text{for all }\xi\in X}$, see e.g.~\cite{risler:nullstellensatz}
  or \cite[Thm.~4.1.4]{bcr:realAlgebraicGeometry}.

  Now let $\hat a = \hat a_r + \I \hat a_i \in \A$ with $\hat a_r,\hat a_i \in \A_\hermitian$. If $\hat a \neq 0$,
  then at least one of $\hat a_r$ or $\hat a_i$ is a non-zero element of $\A_\hermitian$, which we denote by $\hat a'$.
  There also is $\hat p' \in \RR[x_1,\dots,x_n]$
  such that $\pr(\hat p') = \hat a'$. From $\hat a' \neq 0$ it follows that $\hat p' \notin \ker(\pr)$,
  so there is $\xi \in X$ such that $\hat p'(\xi) \neq 0$. We can now define the algebra morphism
  $\varphi \colon \A_\hermitian \to \RR \subseteq \CC$,
  \begin{equation*}
    a \mapsto \varphi(a) \coloneqq p(\xi)\quad\quad\text{for any $p\in\RR[x_1,\dots,x_n]$ such that $\pr(p) = a$}
    ,
  \end{equation*}
  which is welldefined because $\xi \in X$ means that $p(\xi) = 0$ for all $p\in\ker(\pr)$. By construction
  $\varphi(\hat a') = \hat p'(\xi) \neq 0$, so $\varphi_\CC \colon \A \to \CC$, $a\mapsto \varphi_\CC(a) \coloneqq
  \varphi(a_r)+\I \varphi(a_i)$ is a $^*$\=/character of $\A$ such that $\varphi_\CC(\hat a) \neq 0$.
\end{proof}
As a straightforward consequence of Theorem~\ref{theorem:NullstellensatzCommutative},
the real Nullstellensatz holds for every abelian finite-dimensional Lie $^*$\=/algebra.
%
%
To the best of our knowledge, there are no results concerning the real Nullstellensatz for
non-abelian Lie $^*$\=/algebras, except the predecessor article
\cite{schmitt.schoetz:realNullstellensatzForTwoStepNilpotentLieAlgebras}, which solved the $2$\=/step nilpotent case.
The purpose of this article is to give the solution in the general nilpotent case:

\begin{theorem} \label{theorem:main}
  The real Nullstellensatz holds for every nilpotent finite-dimensional Lie $^*$\=/algebra.
\end{theorem}
As an immediate consequence of Theorem~\ref{theorem:main} and Proposition~\ref{proposition:mainCorollary}
one obtains the corollary announced in the introduction. In the proof of Theorem~\ref{theorem:main} we will
reduce the problem for nilpotent finite-dimensional Lie $^*$\=/algebras to the classical real Nullstellensatz
in the form of Theorem~\ref{theorem:NullstellensatzCommutative}.

In the following Sections \ref{sec:weyl} to \ref{sec:realprime} we develop some general theory that will be needed
in the proof of Theorem~\ref{theorem:main}, which is then given in Section~\ref{sec:proof}.
Let us outline the main steps for an enveloping $^*$\=/algebra $(\A,\iota)$ of a nilpotent finite-dimensional
Lie $^*$\=/algebra $\lie g$:
Every real ideal of $\A$ is the intersection of finitely many real prime ideals,
which allows us to reduce to the case of real prime ideals, and after dividing out a real prime ideal and localizing
at the center (without $0$) one obtains a formally real $^*$\=/algebra $\B$ whose center $\KK$ is a field extension of $\CC$
(this is discussed in Section~\ref{sec:realprime}).
By adapting some of the well-known structure theory for universal enveloping algebras of nilpotent finite-dimensional Lie algebras
to the $^*$\=/algebraic setting we show that $\B$ is isomorphic to a tensor product of $\KK$ with a Weyl $^*$\=/algebra
$\Weyl{r}$, see Section~\ref{sec:centralEnv}.
The classical real Nullstellensatz in the form of Theorem~\ref{theorem:NullstellensatzCommutative} can be applied to
a suitable finitely generated $^*$\=/subalgebra $\KK_\fin$ of $\KK$ and provides many $^*$\=/characters $\varphi \colon \KK_\fin \to \CC$.
On the other tensor factor $\Weyl{r}$, the usual Schrödinger $^*$\=/representation can be applied. In combination
we obtain many $^*$\=/representations of $\A$.
In the next Section~\ref{sec:weyl} we develop the machinery to show that these are integrable and irreducible.

\section{Schrödinger \texorpdfstring{$^*$\=/representation}{*-representation} of the Weyl  \texorpdfstring{$^*$\=/algebra}{*-algebra}}
\label{sec:weyl}
The \emph{Heisenberg Lie $^*$\=/algebra of rank $r\in \NN_0$} is the Lie $^*$\=/algebra $\Heisenberg{r}$
of dimension $2r+1$ and with basis $c,d_1,\dots,d_r,e_1,\dots,e_r \in \Heisenberg{r}$ such that
\begin{align}
  c&=c^* &&\text{and} &\kom{c}{x} &= 0 &&\text{for all }x\in \Heisenberg{r},\\
  d_j &= d_j^* &&\text{and} &\kom{d_j}{d_k} &= 0 &&\text{for all }j,k\in\{1,\dots,r\},\\
  e_j &= e_j^* &&\text{and} &\kom{e_j}{e_k} &= 0 &&\text{for all }j,k\in\{1,\dots,r\},
\shortintertext{and}
  &&&& \kom{d_j}{e_k} &= -\I \delta_{j,k} c &&\text{for all }j,k\in\{1,\dots,r\}
  ,
\end{align}
where $\delta_{j,k} \in \{0,1\}$ is the Kronecker-$\delta$. Clearly, $\Heisenberg{r}$ is nilpotent;
in fact, $\kom{{\kom{x}{y}}}{z} = 0$ for all $x,y,z\in \Heisenberg{r}$.

The \emph{Weyl $^*$\=/algebra of rank $r\in \NN_0$} is the $^*$\=/algebra that is generated by $2r$
hermitian elements $p_1,\dots,p_r,q_1,\dots, q_r$ subject to the \emph{canonical commutation relations}
\begin{equation}
  \label{eq:ccr}
  \kom{p_j}{p_k} = \kom{q_j}{q_k} = 0\quad\quad\text{and}\quad\quad
  \kom{p_j}{q_k} = -\I \delta_{j,k}
  \quad\quad\text{for all }j,k\in\{1,\dots,r\},
\end{equation}
where $\delta_{j,k} \in \{0,1\}$ is the Kronecker-$\delta$.
Together with the Lie $^*$\=/algebra morphism $\iotaHeisenberg \colon \Heisenberg{r} \to \Weyl{r}$
that is uniquely determined by
\begin{equation}
  \label{eq:iotaCan}
  \iotaHeisenberg(c) \coloneqq 1,
  \quad\quad
  \iotaHeisenberg(d_j) \coloneqq p_j,
  \quad\quad\text{and}\quad\quad
  \iotaHeisenberg(e_j) \coloneqq q_j
  \quad\quad
  \text{for $j\in\{1,\dots,r\}$,}
\end{equation}
the Weyl $^*$\=/algebra is an enveloping $^*$\=/algebra of $\Heisenberg{r}$.

The \emph{Schrödinger $^*$\=/representation} of $\Weyl{r}$ is the $^*$\=/representation
$\bigl( \Schwartz(\RR^r), \pi_\cat \bigr)$ on the Schwartz space $\Schwartz(\RR^r)$ of rapidly decreasing
smooth functions $f\colon \RR^r \to \CC$, with inner product
\begin{equation}
  \skal{f}{g} \coloneqq \int_{\RR^r} \cc{f(x)}\,g(x)\,\D^rx
  \quad\quad\text{(where $\cc{\argument} \colon \CC \to \CC$ denotes complex conjugation),}
\end{equation}
and with $\pi_\cat \colon \Weyl{r} \to \Adbar\bigl( \Schwartz(\RR^r) \bigr)$ defined as the unique $^*$\=/algebra
morphism that fulfills $\pi_\cat(p_j) \coloneqq P_j$ and $\pi_\cat(q_j) \coloneqq Q_j$ for all $j\in\{1,\dots,r\}$,
where $P_j,Q_j \colon \Schwartz(\RR^r) \to \Schwartz(\RR^r)$ are the \emph{momentum operators} and \emph{position operators},
respectively, i.e.,
\begin{equation}
  P_j \,\psi = -\I \frac{\partial \psi}{\partial x_j}
  \quad\quad\text{and}\quad\quad
  \bigl( Q_j \,\psi\bigr)(x) = x_j \,\psi(x)\quad\quad\text{for all}~x\in \RR^r, \psi \in \Schwartz(\RR^r)
  .
\end{equation}
One readily checks that these are hermitian operators and that, analogously to \eqref{eq:ccr},
\begin{equation}
  \label{eq:CCR}
  \kom{P_j}{P_k} = \kom{Q_j}{Q_k} = 0\quad\quad\text{and}\quad\quad
  \kom{P_j}{Q_k} = -\I \delta_{j,k}
  \quad\quad\text{for all }j,k\in\{1,\dots,r\}.
\end{equation}
We will make use of the following properties of the Schrödinger $^*$\=/representation:

\begin{proposition} \label{proposition:Schroedinger}
  Let $r\in \NN_0$.
  \begin{enumerate}
    \item \label{item:Schroedinger:faithful}
      The Schrödinger $^*$\=/representation $\bigl(\Schwartz(\RR^r),\pi_\cat\bigr)$ of $\Weyl{r}$ is faithful.
    \item \label{item:Schroedinger:irredintegrable}
      The Schrödinger $^*$\=/representation $\bigl(\Schwartz(\RR^r),\pi_\cat\bigr)$ is an irreducible integrable
      $^*$\=/representation of the enveloping $^*$\=/algebra $(\Weyl{r},\iotaHeisenberg)$ of the Heisenberg
      Lie $^*$\=/algebra $\Heisenberg{r}$.
    \item \label{item:Schroedinger:denseOrbit}
      Consider any $\Omega \in \domAna(\pi_\cat,\iotaHeisenberg) \setminus \{0\}$.
      Then $\set{\pi_\cat(a)\,\Omega}{a\in\Weyl{r}}$ is a $\norm{\argument}$\=/dense linear subspace of $\Schwartz(\RR^r)$.
  \end{enumerate}
\end{proposition}
\begin{proof}
  See, e.g., the references
  \cite[Exs.~2.5.2,10.2.14]{schmuedgen:UnboundedOperatorAlgebraAndRepresentationTheory} or
  \cite[Exs.~4.32, 9.33 and Prop.~8.2]{schmuedgen:invitationToStarAlgebras}
  for parts~\ref{item:Schroedinger:faithful} and \ref{item:Schroedinger:irredintegrable}.

  For part~\ref{item:Schroedinger:denseOrbit}, note that $\Dom \coloneqq \set{\pi_\cat(a)\,\Omega}{a\in\Weyl{r}}$ is a
  $\pi_\cat$\=/invariant linear subspace of $\domAna(\pi_\cat,\iotaHeisenberg)$, so we can restrict the Schrödinger
  $^*$\=/representation to $\Dom$ and obtain
  a $^*$\=/representation $(\Dom,\pi_\cat\at{\Dom})$ of $\Weyl{r}$. Clearly $\domAna(\pi_\cat\at{\Dom},\iotaHeisenberg) = \Dom$.
  We identify the $\norm{\argument}$\=/completion $\Hilb$ of $\Dom$ with the $\norm{\argument}$\=/closure of $\Dom$ in the
  $\norm{\argument}$\=/completion $\Ell^2(\RR^r)$ of $\Schwartz(\RR^r)$. Let $P \colon \Ell^2(\RR^r) \to \Ell^2(\RR^r)$
  be the orthogonal projection onto this closed linear subspace $\Hilb$ of $\Ell^2(\RR^r)$. First we show that
  $P\,\phi \in \Schwartz(\RR^r)$ for all $\phi \in \Schwartz(\RR^r)$
  and $\pi_\cat(a)\,P\,\phi = P\,\pi_\cat(a)\,\phi$ for all $a\in \Weyl{r}$, $\phi \in \Schwartz(\RR^r)$:

  Recall that $\Dom \subseteq \dom\bigl(\pi_\cat\at{\Dom}^\dagger\bigr) \subseteq \Hilb$.
  By Proposition \ref{proposition:exponentiation}, there is an integrable (and in particular, selfadjoint)
  $^*$\=/representation
  $\bigl( \dom(\pi_\cat\at{\Dom}^\dagger), \pi_\cat\at{\Dom}^\dagger \bigr)$ of $(\Weyl{r},\iotaHeisenberg)$.
  Consider any $\phi \in \Schwartz(\RR^r)$. Then for all $a\in \A$ and $\psi \in \Dom$, the identity
  \begin{equation*}
    \skal[\big]{P\,\phi}{\pi_\cat\at{\Dom}(a^*)\,\psi}
    =
    \skal[\big]{\phi}{\pi_\cat\at{\Dom}(a^*)\,\psi}
    =
    \skal[\big]{\pi_\cat(a)\,\phi}{\psi}
    =
    \skal[\big]{P\,\pi_\cat(a)\,\phi}{\psi}
  \end{equation*}
  holds, because $\psi \in \Dom \subseteq \Hilb$ and
  $\pi_\cat\at{\Dom}(a^*)\,\psi = \pi_\cat(a^*)\,\psi \in \Dom \subseteq \Hilb$. This shows that
  $P\,\phi \in \dom\bigl(\pi_\cat\at{\Dom}^\dagger\bigr)$ and $\pi_\cat\at{\Dom}^\dagger(a)\,P\,\phi = P\,\pi_\cat(a)\,\phi$
  for all $a\in \Weyl{r}$.
  Consider again any $\phi \in \Schwartz(\RR^r)$. Then for all $a\in \Weyl{r}$ and $\psi \in \Schwartz(\RR^r)$, the identity
  \begin{equation*}
    \skal[\big]{P\,\phi}{\pi_\cat(a^*)\,\psi}
    =
    \skal[\big]{P\,\phi}{\pi_\cat\at{\Dom}^\dagger(a^*)\,P\,\psi}
    =
    \skal[\big]{ \pi_\cat\at{\Dom}^\dagger(a)\,P\,\phi}{P\,\psi}
    =
    \skal[\big]{ P\,\pi_\cat(a)\,\phi}{\psi}
  \end{equation*}
  holds, because $\pi_\cat\at{\Dom}^\dagger(a^*)\,P\,\psi = P\,\pi_\cat(a^*)\,\psi$ and
  $\pi_\cat\at{\Dom}^\dagger(a)\,P\,\phi = P\,\pi_\cat(a)\,\phi$.
  This shows that $P\,\phi \in \dom(\pi_\cat^\dagger)$
  and $\pi_\cat^\dagger(a)\,P\,\phi = P\,\pi_\cat(a)\,\phi$ for all $a\in \Weyl{r}$.
  As $\bigl( \Schwartz(\RR^r), \pi_\cat \bigr)$ is an integrable $^*$\=/representation if $(\Weyl{r},\iotaHeisenberg)$,
  and in particular selfadjoint, it follows that indeed
  $P\,\phi \in \Schwartz(\RR^r)$ and $\pi_\cat(a)\,P\,\phi = P\,\pi_\cat(a)\,\phi$ for all $a\in \Weyl{r}$.

  We have shown that the orthogonal projection $P$ can be restricted to
  $P\at{\Schwartz(\RR^r)} \colon \Schwartz(\RR^r) \to \Schwartz(\RR^r)$,
  and that $P\at{\Schwartz(\RR^r)}$ is in the bounded commutant $\pi_\cat(\Weyl{r})'$.
  By irreducibility of the Schrödinger $^*$\=/representation,
  $P\at{\Schwartz(\RR^r)} = \lambda \id_{\Schwartz(\RR^r)}$ for some $\lambda \in \CC$. As $P\at{\Schwartz(\RR^r)}$
  is a projection onto a nonzero linear subspace of $\Schwartz(\RR^r)$ (e.g., $\Omega$ is in the image of $P\at{\Schwartz(\RR^r)}$),
  $P\at{\Schwartz(\RR^r)} = \id_{\Schwartz(\RR^r)}$. It follows that $\Schwartz(\RR^r) \subseteq \Hilb$, and therefore
  $\Dom$ is $\norm{\argument}$\=/dense in $\Schwartz(\RR^r)$.
\end{proof}
However, $\Weyl{r}$ can also be viewed as an enveloping $^*$\=/algebra of other nilpotent Lie $^*$\=/algebras $\lie g$
by choosing other Lie $^*$\=/algebra morphisms $\iota \colon \lie g \to \Weyl{r}$.
In this section we are going to show that the Schrödinger $^*$\=/representation remains an irreducible integrable
$^*$\=/representation of $( \Weyl{r}, \iota )$ for many other possible choices of $\iota$.

\begin{definition} \label{definition:WeylNil}
  Let $r \in \NN_0$. We write
  \begin{equation}
    \WeylNil_r
    \coloneqq
    \set[\bigg]{
      a_0 + \sum_{j=1}^r p_ja_j
    }{
      a_j\in \genSAlg{q_{j+1},\dots,q_r} \text{ for } j\in\{0,\dots,r\}
    }
    ,
  \end{equation}
  where $ \genSAlg{q_{j+1},\dots,q_r}$ is the $^*$\=/subalgebra of $\Weyl{r}$ that is generated by $q_{j+1},\dots,q_r$,
  i.e., the set of linear combinations of arbitrary products (including the empty product $1$) of $q_{j+1},\dots,q_r$.
\end{definition}
Clearly, $\WeylNil_r$ is a linear subspace of $\Weyl{r}$.
Note also that $a_j p_j = p_j a_j$ for all $j\in \{1,\dots,r\}$ and $a_j\in \genSAlg{q_{j+1},\dots,q_r}$;
therefore, $\WeylNil_r$ is stable under the $^*$\=/involution.
One can also check that $\WeylNil_r$ is a (locally nilpotent) Lie $^*$\=/subalgebra of $\Weyl{r}$,
but we won't make use of this fact as we are only interested in finite-dimensional Lie $^*$\=/subalgebras
of $\Weyl{r}$ that are contained in $\WeylNil_r$.

\begin{example}
  Let $\lie w_5$ be the $5$-dimensional truncated Witt Lie $^*$\=/algebra, i.e., $\lie w_5$ has basis $e_1,\dots,e_5$,
  and $e_k^* = e_k$ and $\kom{e_k}{e_\ell} \coloneqq \I (\ell-k) e_{k+\ell}$ for all $k,\ell \in \{1,\dots, 5\}$,
  where we set $e_j \coloneqq 0$ for $j>5$.
  Let $\iota \colon \lie w_5 \to \Weyl{2}$ be the linear map that is determined by
  \begin{align*}
    e_1 \mapsto \iota(e_1) &\coloneqq -\sqrt{6} p_1 \\
    e_2 \mapsto \iota(e_2) &\coloneqq \frac{1}{6} q_1^3 + q_1 q_2 - p_2 \\
    e_3 \mapsto \iota(e_3) &\coloneqq \sqrt{\frac32} q_1^2 + \sqrt{6} q_2 \\
    e_4 \mapsto \iota(e_4) &\coloneqq 3q_1 \\
    e_5 \mapsto \iota(e_5) &\coloneqq \sqrt{6}.
  \end{align*}
  One can check that $\iota(e_k)^* = \iota(e_k)$ and
  $\kom{\iota(e_k)}{\iota(e_\ell)} = \iota\bigl( \kom{e_k}{e_\ell} \bigr)$ for all $k,\ell \in \{1,\dots,5\}$,
  so $\iota$ is a Lie $^*$\=/algebra morphism. It is also not hard to check that
  $\Weyl{2}$ is generated (as a $^*$\=/algebra) by $\iota(\lie w_5)$ (it suffices to check the generators $p_1,p_2,q_1,q_2$),
  i.e., $(\Weyl{2},\iota)$
  is an enveloping $^*$\=/algebra of $\lie{w}_5$. By the following Theorem~\ref{theorem:integrable},
  the Schrödinger $^*$\=/representation $\bigl( \Schwartz(\RR^2), \pi_\cat \bigr)$
  is an irreducible integrable $^*$\=/representation of $(\Weyl{2},\iota)$.
\end{example}

\begin{theorem} \label{theorem:integrable}
  Let $r\in \NN_0$ and let $\bigl(\Weyl{r}, \iota\bigr)$ be an enveloping $^*$\=/algebra of
  some finite-dimensional Lie $^*$\=/algebra $\lie g$.
  Assume that $\iota(\lie g) \subseteq \WeylNil_r$. Then the Schrödinger $^*$\=/representation
  $\bigl( \Schwartz(\RR^r), \pi_\cat \bigl)$ is an irreducible integrable $^*$\=/representation
  of $\bigl(\Weyl{r}, \iota\bigr)$.
\end{theorem}
The $r=0$\,-case is clear and we will only discuss the $r>0$\,-case in the following.

Note that it is well known that (up to unitary equivalence) every irreducible integrable $^*$\=/representation
of an enveloping $^*$\=/algebra of a nilpotent finite-dimensional Lie $^*$\=/algebra can be described
by polynomial differential operators acting on a Schwartz space $\Schwartz(\RR^m)$ for some $m\in\NN_0$,
see e.g.~\cite[Sec.~5]{pedersen:geometricQuantizationAndTheUniversalEnvelopingAlgebraOfANilpotentLieGroup}.
Theorem~\ref{theorem:integrable} provides a partial converse, guaranteeing integrability of certain $^*$\=/representations
by polynomial differential operators. The proof will occupy the rest of this section,
throughout which we keep the notation and assumptions of Theorem~\ref{theorem:integrable}.
We also use standard multiindex notation, namely, we write
\begin{equation}
  Q^k \coloneqq \prod_{j=1}^r Q_j^{k_j} \in \Adbar\bigl( \Schwartz(\RR^r) \bigr)\quad\quad\text{for }k\in\NN_0^r.
\end{equation}
The following definitions and intermediary results are only relevant for the proof of Theorem~\ref{theorem:integrable},
i.e., until the end of this section.

\begin{definition}
  Let $m\in \NN$. We write $\SchroedNilGen_m \coloneqq \bigcup_{t=0}^r \SchroedNilGen_{m;t}$
  where, for $t\in \{1,\dots,r\}$,
  \begin{equation}
    \SchroedNilGen_{m;t}
    \coloneqq
    \set[\Big]{
      \!P_t Q^k
    }{
      k \in \NN_0^r\text{ such that } k_j = 0 \text{ for }j\in\{1,\dots,t\}\text{ and }k_j \le m\text{ for }j \in \{t+1,\dots,r\}\!
    },\!
  \end{equation}
  and similarly,
  \begin{equation}
    \SchroedNilGen_{m;0}
    \coloneqq
    \set[\Big]{ \!Q^k }{ k \in \NN_0^r\text{ such that } k_j \le m\text{ for }j\in\{1,\dots,r\}\! }
    .
  \end{equation}
\end{definition}
Clearly $\SchroedNilGen_m \subseteq \SchroedNilGen_{m'}$ for all $m,m' \in \NN$ with $m\le m'$, and
$\bigcup_{m\in \NN} \SchroedNilGen_m$ is a basis of $\pi_\cat(\WeylNil_r)$.
It follows that for sufficiently large $m \in \NN$, the finite-dimensional linear subspace
$\pi_\cat(\iota(\lie g))$ of $\pi_\cat(\WeylNil_r)$ is contained in the linear span of $\SchroedNilGen_m$.

\begin{definition}
  We choose $n$ as the smallest element of $\NN$ such that
  $\pi_\cat(\iota(\lie g))$ is contained in the linear span of $\SchroedNilGen_n$.
  Then, we define the function $\Omega \in \Schwartz(\RR^r)$,
  \begin{equation}
    \RR^r \ni x \mapsto \Omega(x) \coloneqq
    \prod_{j=1}^r \exp\biggl( - \frac{x_j^{(2n)^j}}{2} \biggr)
    ,
  \end{equation}
  the (finite) sets of multiindices
  \begin{equation}
    K_d
    \coloneqq
    \set[\bigg]{k \in \NN_0^r}{ \sum_{j=1}^r \frac{k_j}{(2n)^j} \le d }
    \quad\quad\text{for }d\in \NN_0
    ,
  \end{equation}
  and the linear subspace
  \begin{equation}
    \Dom_d
    \coloneqq
    \set[\bigg]{
      \sum_{k \in K_d} f_k \,Q^k\, \Omega
    }{
      f_k \in \CC \text{ for }k\in K_d
    }
    \subseteq
    \Schwartz(\RR^r)
    \quad\quad\text{for }d\in \NN_0
    .
  \end{equation}
  We also define
  \begin{equation}
    \Dom_\infty \coloneqq \bigcup_{d\in \NN_0} \Dom_d \subseteq \Schwartz(\RR^r)
    .
  \end{equation}
\end{definition}
Clearly $K_d \subseteq K_{d'}$ for $d,d' \in \NN_0$ with $d\le d'$, so $(\Dom_d)_{d\in \NN_0}$
is an increasing sequence of linear subspaces of $\Schwartz(\RR^r)$ and therefore $\Dom_\infty$
is also a linear subspace of $\Schwartz(\RR^r)$.
This space $\Dom_\infty$ consists of all products of a polynomial function on $\RR^r$ with $\Omega$.
In particular, every element $f\in\Dom_\infty$ can be expanded in a unique way as
$f = \sum_{k\in \NN_0^r} f_k\,Q^k\,\Omega$ with coefficients $f_k \in \CC$ (and only finitely many of them non-zero).
Note that, $t\in \{1,\dots,r\}$,
\begin{align}
  \bigl(P_t\,\Omega\bigr)(x)
  =
  -\I \frac{\partial}{\partial x_t}
  \prod_{j=1}^r \exp\biggl( - \frac{x_j^{(2n)^j}}{2} \biggr)
  &=
  \frac{\I (2n)^t \,x_t^{(2n)^t-1}}{2}
  \prod_{j=1}^r \exp\biggl( - \frac{x_j^{(2n)^j}}{2} \biggr)
  \quad\quad\text{for all } x\in\RR^r,
\shortintertext{so}
  \label{eq:PtOmega}
  P_t\,\Omega
  &=
  \frac{\I (2n)^t}{2} \,Q_t^{(2n)^t-1}\,\Omega
  .
\end{align}
Consider any $m\in \NN$. By substituting $y \coloneqq x^{2m}$ one finds that
\begin{equation}
  \int_{-\infty}^\infty x^{2k} \exp\bigl(-x^{2m} \bigr)\,\D x
  =
  \frac{1}{m} \GammaFun\biggl( \frac{2k+1}{2m} \biggr)
  \le
  \GammaFun\biggl( \frac{2k+1}{2m} \biggr)
  \quad\quad\text{for all }k\in \NN_0
  .
\end{equation}
Consequently,
\begin{equation}
  \label{eq:integral}
  \norm[\big]{Q^k\,\Omega}
  =
  \prod_{j=1}^r \biggl( \int_{-\infty}^\infty x_j^{2k_j} \exp\Bigl( - x_j^{(2n)^j}\Bigr)\,\D x_j \biggr)^{\!\frac12}
  \le
  \prod_{j=1}^r \GammaFun\biggl( \frac{2k_j+1}{(2n)^j} \biggr)^{\!\frac12}
  .
\end{equation}

\begin{definition}
  We define the norm $\norm{\argument}_{\approx}$ on $\Dom_\infty$ as
  \begin{equation}
    \norm{f}_{\approx}
    \coloneqq
    \sum_{k \in \NN_0^r} \abs{f_k} \prod_{j=1}^r \GammaFun\biggl( \frac{2k_j+1}{(2n)^j} \biggr)^{\!\frac12}
    \quad\text{for }f = \sum_{k \in \NN_0^r} f_k \,Q^k\, \Omega \in \Dom_\infty \text{ with coefficients $f_k \in \CC$}
    .
  \end{equation}
\end{definition}
By \eqref{eq:integral},
\begin{equation}
  \label{eq:normestimate}
  \norm{f}
  \le
  \sum_{k \in \NN_0^r} \abs{f_k}\, \norm[\big]{ Q^k \,\Omega}
  \le
  \sum_{k \in \NN_0^r} \abs{f_k} \prod_{j=1}^r \GammaFun\biggl( \frac{2k_j+1}{(2n)^j} \biggr)^{\!\frac12}
  =
  \norm{f}_\approx
  \quad\quad\text{for all }f \in \Dom_\infty
  .
\end{equation}
In the following it will be easier to estimate the norm $\norm{\argument}_\approx$ rather than the Hilbert norm $\norm{\argument}$.
As a preparation we need some elementary properties of the Gamma function:

\begin{lemma} \label{lemma:GammaEstimate}
  Let $\beta \in {]0,\infty[}$. For all $x_1,\dots,x_r \in {]0,\beta]}$ and $\lambda_1,\dots,\lambda_r \in [0,1]$
  that fulfill $\sum_{j=1}^r \lambda_j \le 1$, the estimate
  \begin{equation}
    \prod_{j=1}^r \GammaFun\bigl(x_j+2\lambda_j\bigr)
    \le
    (\beta + 1)^2 \prod_{j=1}^r \GammaFun(x_j)
  \end{equation}
  holds.
\end{lemma}
\begin{proof}
  Clearly $\GammaFun(x+2) = (x+1)x\GammaFun(x) \le (\beta+1)^2 \GammaFun(x)$ for all $x\in {]0,\beta]}$,
  and by logarithmic convexity of the Gamma function it follows that
  \begin{equation*}
    \GammaFun(x+2\lambda)
    =
    \GammaFun\bigl(\lambda(x+2) + (1-\lambda)x\bigr)
    \le
    \GammaFun(x+2)^\lambda \GammaFun(x)^{1-\lambda}
    \le
    (\beta+1)^{2\lambda} \GammaFun(x)
    \quad\quad
    \text{for all $x\in {]0,\beta]}$, $\lambda\in [0,1]$.}
  \end{equation*}
  Now let $x_1,\dots,x_r \in {]0,\beta]}$ and $\lambda_1,\dots,\lambda_r \in [0,1]$
  such that $\sum_{j=1}^r \lambda_j \le 1$. Then
  \begin{equation*}
    \prod_{j=1}^r \GammaFun\bigl(x_j+2\lambda_j\bigr)
    \le
    \prod_{j=1}^r (\beta+1)^{2\lambda_j} \Gamma(x_j)
    =
    (\beta+1)^{2\sum_{j=1}^r \lambda_j} \prod_{j=1}^r \Gamma(x_j)
    \le
    (\beta+1)^2 \prod_{j=1}^r \Gamma(x_j)
    .
  \end{equation*}
\end{proof}

\begin{proposition} \label{proposition:OmegaAnalytic}
  Let $d\in \NN_0$, $f \in \Dom_d$, and $G \in \SchroedNilGen_n$. Then
  \begin{equation}
    G f \in \Dom_{d+1}
    \quad\quad
    \text{and}
    \quad\quad
    \norm{G f}_\approx \le 3(2n)^r (d+1)\,\norm{f}_\approx
    .
  \end{equation}
\end{proposition}
\begin{proof}
  Let $d\in \NN_0$ and $f \in \Dom_d$; we can expand $f$ as $f = \sum_{k \in K_d} f_k \,Q^k\, \Omega$ with coefficients $f_k \in \CC$.

  First consider $G \in \SchroedNilGen_{n;t}$ for $t\in\{1,\dots,r\}$, so $G = P_t Q^\ell$ for some $\ell \in \NN_0^r$
  such that $\ell_j = 0$ for $j\in\{1,\dots,t\}$ and $\ell_j  \le n$ for $j \in \{t+1,\dots,r\}$. We write
  $e_t \in \NN_0^r$ for the multiindex $(e_t)_j \coloneqq \delta_{t,j}$, where $\delta$ is the Kronecker-$\delta$. Then
  \begin{align*}
    Gf
    &=
    \sum_{k \in K_d} f_k \, P_t \,Q^{k+\ell}\, \Omega
    \\
    &=
    -\I \sum_{\substack{k \in K_d \\ k_t > 0}} k_t\,f_k \,Q^{k+\ell-e_t} \, \Omega
    +
    \sum_{k \in K_d} f_k \,Q^{k+\ell} P_t \, \Omega
    \\
    &=
    -\I \sum_{\substack{k \in K_d \\ k_t > 0}} k_t\,f_k \,Q^{k+\ell-e_t} \, \Omega
    +
    \frac{\I (2n)^t}{2} \sum_{k \in K_d} f_k \,Q^{k+\ell+((2n)^t-1) e_t} \, \Omega
  \end{align*}
  by \eqref{eq:CCR} and \eqref{eq:PtOmega}. We write
  \begin{equation*}
    \lambda_j \coloneqq \frac{\ell_j + ((2n)^t-1)\delta_{t,j}}{(2n)^j} \in {[0,\infty[}
    \quad\quad\text{for }j\in\{1,\dots,r\}
    ,
  \end{equation*}
  then $\sum_{j=1}^r \lambda_j = \frac{(2n)^t-1}{(2n)^t} + \sum_{j=t+1}^r \frac{\ell_j}{(2n)^j} \le 1$ because
  \begin{equation*}
    \sum_{j=t+1}^r \frac{\ell_j}{(2n)^j}
    \le
    \frac{1}{2} \sum_{j=t+1}^r \frac{1}{(2n)^{j-1}}
    =
    \frac{1}{2(2n)^t} \sum_{j'=0}^{r-t-1} \frac{1}{(2n)^{j'}}
    \le
    \frac{1}{(2n)^t}
    .
  \end{equation*}
  It follows that $k+\ell+((2n)^t-1) e_t \in K_{d+1}$ for all $k\in K_d$
  (and then clearly also $k+\ell - e_t \in K_{d+1}$ for all $k\in K_d$ with $k_t>0$), because
  \begin{equation*}
    \sum_{j=1}^r \frac{k_j+\ell_j+((2n)^t-1) \delta_{t,j} }{ (2n)^j }
    =
    \sum_{j=1}^r \frac{k_j}{ (2n)^j } + \sum_{j=1}^r \lambda_j
    \le
    d + 1
    .
  \end{equation*}
  This shows that $Gf \in \Dom_{d+1}$. Moreover:
  \begin{align*}
    \norm{G f}_\approx
    &\le
    \sum_{\substack{k \in K_d \\ k_t > 0}} k_t\,\abs{f_k} \,\norm[\big]{ Q^{k+\ell-e_t} \, \Omega }_\approx
    +
    \frac{(2n)^t}{2} \sum_{k \in K_d} \abs{f_k} \, \norm[\big]{ Q^{k+\ell+((2n)^t-1) e_t} \, \Omega }_\approx
    \\
    &=
    \sum_{\substack{k \in K_d \\ k_t > 0}} k_t\,\abs{f_k}
    \,\prod_{j=1}^r \GammaFun\biggl( \frac{2k_j+2\ell_j-2\delta_{t,j}+1}{(2n)^j} \biggr)^{\!\frac12}
    +
    \frac{(2n)^t}{2} \sum_{k \in K_d} \abs{f_k} \,
    \prod_{j=1}^r \GammaFun\biggl( \frac{2k_j+1}{(2n)^j} + 2\lambda_j\biggr)^{\!\frac12}
    .
  \end{align*}
  Now note that $k_t^2 \le \bigl(2 k_t - 1 + (2n)^t\bigr)(2k_t -1)$ for all $k_t \in \NN$,
  because
  \begin{equation*}
    \bigl(2 k_t - 1 + (2n)^t\bigr)(2k_t -1) - k_t^2 = 3k_t^2 - 4 k_t + 1 + (2n)^t(2 k_t -1) = 3k_t^2 -1 + \bigl((2n)^t-2\bigr)(2 k_t-1) \ge 0
    .
  \end{equation*}
  This estimate and the identities $\ell_t=0$ and $x\,\Gamma(x) = \Gamma(x+1)$ for $x\in {]0,\infty[}$
  allow us to derive the following inequality, which we then use for estimating the first sum in the above bound
  for $\norm{G f}_\approx$:
  \begin{align*}
    k_t \GammaFun\biggl( \frac{2k_t+2\ell_t-2\delta_{t,t}+1}{(2n)^t} \biggr)^{\!\frac12}
    &\le
    (2n)^t
    \biggl( \frac{2 k_t - 1 + (2n)^t}{(2n)^t} \biggr)^{\!\frac12}
    \biggl( \frac{2k_t -1}{(2n)^t} \biggr)^{\!\frac12}
    \GammaFun\biggl( \frac{2k_t+2\ell_t-1}{(2n)^t} \biggr)^{\!\frac12}
    \\
    &=
    (2n)^t \biggl( \frac{2 k_t + 2\ell_t - 1}{(2n)^t} + 1\biggr)^{\!\frac12}
    \biggl( \frac{2 k_t + 2\ell_t - 1}{(2n)^t}\biggr)^{\!\frac12}
    \GammaFun\biggl( \frac{2k_t+2\ell_t-1}{(2n)^t} \biggr)^{\!\frac12}
    \\
    &=
    (2n)^t
    \GammaFun\biggl( \frac{2k_t+2\ell_t + 2((2n)^t-1)\delta_{t,t} + 1}{(2n)^t} \biggr)^{\!\frac12}
  \end{align*}
  and consequently
  \begin{align*}
    \sum_{\substack{k \in K_d \\ k_t > 0}} k_t\,\abs{f_k}
    \,\prod_{j=1}^r \GammaFun\biggl( \frac{2k_j+2\ell_j-2\delta_{t,j}+1}{(2n)^j} \biggr)^{\!\frac12}
    &\le
    (2n)^t \sum_{\substack{k \in K_d \\ k_t > 0}} \abs{f_k}
    \,\prod_{j=1}^r \GammaFun\biggl( \frac{2k_j+2\ell_j+ 2((2n)^t-1)\delta_{t,j}+1}{(2n)^j} \biggr)^{\!\frac12}
    \\
    &=
    (2n)^t \sum_{\substack{k \in K_d \\ k_t > 0}} \abs{f_k}
    \,\prod_{j=1}^r \GammaFun\biggl( \frac{2k_j+1}{(2n)^j} + 2\lambda_j \biggr)^{\!\frac12}
    .
  \end{align*}
  It follows that
  \begin{equation*}
    \norm{G f}_\approx \le \frac{3(2n)^t }{2} \sum_{k \in K_d} \abs{f_k} \prod_{j=1}^r \GammaFun\biggl( \frac{2k_j+1}{(2n)^j} + 2\lambda_j\biggr)^{\frac12}
    .
  \end{equation*}
  As $\sum_{j=1}^r \lambda_j \le 1$ and $\frac{2k_j+1}{(2n)^j} \le 2d+1$ for all $j\in \{1,\dots,r\}$,
  Lemma~\ref{lemma:GammaEstimate} applies with $\beta \coloneqq 2d+1$, so
  \begin{align*}
    \norm{G f}_\approx
    &\le
    \frac{3(2n)^t}{2} \sum_{k \in K_d} \abs{f_k} \,\bigl(2d+2\bigr)
    \prod_{j=1}^r \GammaFun\biggl( \frac{2k_j+1}{(2n)^j}\biggr)^{\!\frac12}
    =
    3(2n)^t(d+1) \norm{f}_\approx
    .
  \end{align*}
  It remains to treat the case $G \in \SchroedNilGen_{n;0}$, i.e., $G = Q^\ell$ for some $\ell \in \NN_0^r$
  such that $\ell_j \le n$ for all $j\in\{1,\dots,r\}$. In this case
  $Gf = \sum_{k \in K_d} f_k \, Q^{k+\ell}\, \Omega$.
  We write $\lambda_j \coloneqq \frac{\ell_j}{(2n)^j} \in {[0,\infty[}$ for $j\in\{1,\dots,r\}$. Then again
  $\sum_{j=1}^r \lambda_j \le \frac{1}{2}\sum_{j=1}^r \frac{1}{(2n)^{j-1}} \le 1$, and $k+\ell \in K_{d+1}$ because
  \begin{equation*}
    \sum_{j=1}^r \frac{k_j+\ell_j}{(2n)^j} = \sum_{j=1}^r \frac{k_j}{(2n)^j} + \sum_{j=1}^r \lambda_j \le d+1
    .
  \end{equation*}
  This shows that $Gf \in \Dom_{d+1}$, and by similar estimates as before,
  in particular by application of Lemma~\ref{lemma:GammaEstimate} with $\beta \coloneqq 2d+1$, it also follows that
  \begin{align*}
    \norm{G f}_\approx
    &\le
    \sum_{k \in K_d} \abs{f_k} \, \norm[\big]{ Q^{k+\ell} \,\Omega }_\approx
    \\
    &=
    \sum_{k \in K_d} \abs{f_k} \, \prod_{j=1}^r \GammaFun\biggl( \frac{2k_j+1}{(2n)^j} + 2\lambda_j \biggr)^{\!\frac12}
    \\
    &\le
    \sum_{k \in K_d} \abs{f_k} \, \bigl( 2d+2 \bigr) \prod_{j=1}^r \GammaFun\biggl( \frac{2k_j+1}{(2n)^j} \biggr)^{\!\frac12}
    \\
    &=
    2(d+1) \norm{f}_\approx
    .
  \end{align*}
  In all cases, i.e., for all $G \in \mathcal{B}_{n;t}$ with $t\in \{0,\dots,r\}$,
  the estimate $\norm{G f}_\approx \le 3(2n)^r(d+1) \norm{f}_\approx$ holds.
\end{proof}

\begin{corollary} \label{corollary:OmegaAnalytic}
  There is $\rho \in {[0,\infty[}$ such that $\norm[\big]{G_\ell\cdots G_1\,\Omega} \le \rho^\ell\,\ell!$
  for all $\ell\in \NN$ and $G_1,\dots,G_\ell \in \SchroedNilGen_n$.
\end{corollary}
\begin{proof}
  Induction over $\ell$, starting with $\Omega \in \Dom_0$ and $\norm{\Omega}_\approx \le \norm{\Omega}_\approx$,
  and applying Proposition~\ref{proposition:OmegaAnalytic} in the induction step, shows that
  \begin{equation*}
    G_\ell\cdots G_1\,\Omega \in \Dom_{\ell}
    \quad
    \text{and}
    \quad
    \norm[\big]{G_\ell\cdots G_1\,\Omega}_\approx \le \bigl(3(2n)^r\bigr)^\ell\ell!\,\norm{\Omega}_\approx
    \quad\text{for all $\ell \in \NN_0$ and $G_1,\dots,G_\ell \in \SchroedNilGen_n$.}
  \end{equation*}
  Set $\rho \coloneqq 3(2n)^r \bigl(\norm{\Omega}_\approx+1\bigr)$, then $\bigl(3(2n)^r\bigr)^\ell \norm{\Omega}_\approx \le \rho^\ell$
  for all $\ell \in \NN$.
  By \eqref{eq:normestimate}, the norm $\norm{\argument}_\approx$ can be used to estimate the Hilbert norm $\norm{\argument}$
  on $\Dom_\infty$ and it follows that
  \begin{equation*}
    \norm[\big]{G_\ell\cdots G_1\,\Omega}
    \le
    \norm[\big]{G_\ell\cdots G_1\,\Omega}_\approx
    \le
    \bigl(3(2n)^r\bigr)^\ell\ell!\,\norm{\Omega}_\approx
    \le
    \rho^\ell\,\ell!
    \quad\quad\!\! \text{for all $\ell \in \NN$ and $G_1,\dots,G_\ell \in \SchroedNilGen_n$}.\!
  \end{equation*}
\end{proof}

\begin{proof}[of Theorem~\ref{theorem:integrable}]
  From part~\ref{item:Schroedinger:irredintegrable} of Proposition~\ref{proposition:Schroedinger}
  we already know that the Schrödinger $^*$\=/representation
  $\bigl(\Schwartz(\RR^r),\pi_\cat\bigr)$ of $\Weyl{r}$ is selfadjoint and irreducible
  (these properties depend only on the $O^*$\=/algebra $\pi_\cat(\Weyl{r})$, not on whether or how we view $\Weyl{r}$
  as an enveloping $^*$\=/algebra of some Lie $^*$\=/algebra).
  It only remains to show that $\domAna(\pi_\cat,\iota)$ is $\norm{\argument}$\=/dense in $\Schwartz(\RR^r)$.

  So let $B$ be a basis of $\pi_\cat(\iota(\lie g))$.
  As $\pi_\cat(\iota(\lie g))$ is contained in the linear span of $\SchroedNilGen_n$ by definition of $n$,
  every element $H\in B$ can be expanded as
  \begin{equation*}
    H = \sum_{G \in \SchroedNilGen_n} \xi_{H,G}\,G
    \quad\quad\text{with suitable coefficients $\xi_{H,G} \in \CC$.}
  \end{equation*}
  As both $B$ and $\SchroedNilGen_n$ are finite,
  the maximum $\mu \coloneqq \max \set[\big]{\abs{\xi_{H,G}}}{H\in B \text{ and } G \in\SchroedNilGen_n}$ exists.
  For every $\ell \in \NN$ and all $H_1,\dots,H_\ell \in B$ the estimate
  \begin{equation*}
    \norm[\big]{ H_\ell \cdots H_1\,\Omega }
    \le
    \sum_{G_1,\dots,G_\ell \in \SchroedNilGen_n}
    \abs[\big]{\xi_{H_\ell,G_\ell}} \cdots \abs[\big]{\xi_{H_1,G_1}}
    \,
    \norm[\big]{ G_\ell \cdots G_1\,\Omega }
    \le
    \bigl(\mu \rho \abs{\SchroedNilGen_n} \bigr)^\ell\,\ell!
  \end{equation*}
  holds, with $\rho \in {[0,\infty[}$ from the previous Corollary~\ref{corollary:OmegaAnalytic}
  and with $\abs{\SchroedNilGen_n}$ the number of elements in $\SchroedNilGen_n$.
  This shows that $\Omega$ is an analytic vector for the $^*$\=/representation
  $\bigl( \Schwartz(\RR^r), \pi_\cat \bigl)$ of the enveloping $^*$\=/algebra $\bigl(\Weyl{r}, \iota\bigr)$ of $\lie g$,
  i.e., $\Omega \in \domAna(\pi_\cat,\iota)$. By $\pi_\cat$-invariance of $\domAna(\pi_\cat,\iota)$,
  it follows that $\set{\pi_\cat(a)\,\Omega}{a\in \Weyl{r}} \subseteq \domAna(\pi_\cat,\iota)$.

  However, $\Omega$ is also an analytic vector with respect to the enveloping $^*$\=/algebra
  $\bigl(\Weyl{r}, \iotaHeisenberg\bigr)$ of the Heisenberg Lie $^*$\=/algebra $\Heisenberg{r}$ as in \eqref{eq:iotaCan},
  i.e., $\Omega \in \domAna(\pi_\cat,\iotaHeisenberg)$:
  Namely, let $c,d_1,\dots,d_r, e_1,\dots,e_r$ be the standard basis of $\Heisenberg{r}$,
  then $\pi_\cat(\iotaHeisenberg(c)) = \id_{\Schwartz(\RR^r)} = Q^0 \in \SchroedNilGen_1$ and
  $\pi_\cat(\iotaHeisenberg(d_j)) = P_j \in \SchroedNilGen_1$, $\pi_\cat(\iotaHeisenberg(e_j)) = Q_j \in \SchroedNilGen_1$.
  As $\SchroedNilGen_1 \subseteq \SchroedNilGen_n$, Corollary~\ref{corollary:OmegaAnalytic} also
  shows that $\Omega \in \domAna(\pi_\cat,\iotaHeisenberg)$.
  Part~\ref{item:Schroedinger:denseOrbit} of Proposition~\ref{proposition:Schroedinger} therefore shows that
  $\set{\pi_\cat(a)\,\Omega}{a\in \Weyl{r}}$ is $\norm{\argument}$\=/dense in $\Schwartz(\RR^r)$.
  Consequently, $\domAna(\pi_\cat,\iota)$ is $\norm{\argument}$\=/dense in $\Schwartz(\RR^r)$.
\end{proof}

\section{Central enveloping \texorpdfstring{$^*$\=/algebras}{*-algebras} of nilpotent Lie \texorpdfstring{$^*$\=/algebras}{*-algebras}}
\label{sec:centralEnv}
All tensor products $\otimes$ are tensor products of complex vector spaces.
For the \emph{center} of a $^*$\=/algebra $\A$ we write
$\Zentrum(\A)\coloneqq \set[\big]{a\in \A}{aa'=a'a~\text{for all}~a'\in\A}$.
Clearly, $\Zentrum(\A)$ is a commutative $^*$\=/subalgebra of $\A$.
Now consider two $^*$\=/algebras $\A$ and $\B$. Their tensor product $\A \otimes \B$, with multiplication and $^*$\=/involution
defined as $(a \otimes b)(a' \otimes b') \coloneqq aa' \otimes bb'$ and $(a\otimes b)^* \coloneqq a^* \otimes b^*$
for all $a,a' \in \A$ and $b,b'\in \B$, is again a $^*$\=/algebra.
If $\C$ is a third $^*$\=/algebra and $\Phi \colon \A \to \C$ and $\Psi \colon \B \to \C$ two $^*$\=/algebra morphisms
such that $\kom{\Phi(a)}{\Psi(b)} = 0$ for all $a\in \A$, $b\in \B$, then the linear map
$\A \otimes \B \to \C$,
\begin{equation}
  a \otimes b \mapsto \Phi(a)\,\Psi(b)
\end{equation}
is again a $^*$\=/algebra morphism, which is easy to check.
Note also that the center of the tensor product $^*$\=/algebra is $\Zentrum(\A \otimes \B) = \Zentrum(\A) \otimes \Zentrum(\B)$
(this is well known: simply expand any central element of $\A \otimes \B$ as
$\sum_{i=1}^k a_i \otimes b_i$ with $a_1,\dots, a_k \in \A$ linearly independent and $b_1,\dots, b_k \in \B$,
and consider commutators with $1 \otimes b'$ for arbitrary $b' \in \B$ to show that $b_1,\dots,b_k \in \Zentrum(\B)$;
in a second step, choose another expansion of the form
$\sum_{i=1}^k a_i \otimes b_i$ with $b_1,\dots, b_k \in \Zentrum(\B)$ linearly independent and $a_1,\dots, a_k \in \A$,
and consider commutators with $a' \otimes 1$ for arbitrary $a'\in \A$ to show that $a_1,\dots,a_k \in \Zentrum(\A)$).

We will repeatedly make use of an elementary observation (Lemma~\ref{lemma:tensorWeyl})
regarding tensor products with the Weyl $^*$\=/algebra $\Weyl{1}$.
For simplicity, we write $p \coloneqq p_1$ and $q \coloneqq q_1$ for the generators of $\Weyl{1}$.
It is well known that $\set{p^k q^\ell}{k,\ell \in \NN_0}$ is a basis of $\Weyl{1}$.
Therefore, if $\A$ is a $^*$\=/algebra and $x\in \A \otimes \Weyl{1}$, then
$x$ can be expressed as $x = \sum_{k,\ell = 0}^\infty x_{k,\ell} \otimes p^k q^\ell$ with uniquely determined coefficients
$x_{k,\ell} \in \A$, of which only finitely many are non-zero.

\begin{lemma} \label{lemma:tensorWeyl}
  Let $\A$ be a $^*$\=/algebra and $\mathcal S$ a linear subspace of $\A \otimes \Weyl{1}$ such that
  $(1\otimes b)s,s(1\otimes b) \in \mathcal{S}$ for all $b\in \Weyl{1}$, $s \in \mathcal S$.
  Suppose $\hat s \in \mathcal S$. Then the coefficients $\hat s_{k,\ell} \in \A$ of $\hat s$ fulfill
  $\hat s_{k,\ell} \otimes 1 \in \mathcal S$ for all $k,\ell \in \NN_0$.
\end{lemma}
\begin{proof}
  We repeat the standard proof for convenience of the reader:
  For any $b \in \Weyl{1}$, we define the inner derivation $\ad_{1\otimes b} \colon \A \otimes \Weyl{1} \to \A \otimes \Weyl{1}$,
  \begin{equation*}
    x \mapsto \ad_{1\otimes b} x \coloneqq \kom{1 \otimes b}{x}
    .
  \end{equation*}
  Then $\ad_{1\otimes b} s \in \mathcal S$ for all $s\in \mathcal S$ by assumption.
  We also define certain products of such derivations, namely
  $\ad^{k,\ell} \coloneqq \frac{(-\I)^k \I^\ell }{k!\,\ell!} \ad_{1 \otimes q}^k \ad_{1 \otimes p}^\ell$
  for $k,\ell \in \NN_0$; again $\ad^{k,\ell} s \in \mathcal S$ for all $s\in \mathcal S$. Let $a \in \A$. From
  \begin{align*}
    && -\I \ad_{1\otimes q} ( a\otimes p) &=  a\otimes 1, & -\I \ad_{1\otimes q} ( a\otimes q) &= 0,&&\\
    && \I \ad_{1\otimes p} ( a\otimes p) &= 0, & \I \ad_{1\otimes p} ( a\otimes q) &=  a \otimes 1 &&
  \end{align*}
  it follows that $\ad^{k,\ell} ( a \otimes p^m q^n) = 0$ for all $m,n\in \NN_0$ with $m<k$ or $n < \ell$,
  and $\ad^{k,\ell} ( a \otimes p^k q^\ell) = a \otimes 1$.

  Now consider $\hat s \in \mathcal S$.
  If $t\in \NN_0$ is sufficiently large, then $\hat s_{k,\ell} = 0$ for all $k,\ell \in \NN_0$ with $k+\ell > t$,
  and in particular $\hat s_{k,\ell} \otimes 1 \in \mathcal S$.
  Let $t\in \NN_0$ and assume it has already been shown that
  $\hat s_{k,\ell} \otimes 1 \in \mathcal S$ for all $k,\ell \in \NN_0$ with $k+\ell > t$. We show that then also
  $\hat s_{k,\ell} \otimes 1 \in \mathcal S$ for all $k,\ell \in \NN_0$ with $k+\ell = t$: namely,
  \begin{equation*}
    \sum_{\substack{m,n \in \NN_0 \\ m+n \le t}} \hat s_{m,n} \otimes p^m q^n
    =
    \hat s - \sum_{\substack{m,n \in \NN_0 \\ m+n > t}} (\hat s_{m,n} \otimes 1)(1 \otimes p^m q^n) \in \mathcal S
    ,
  \end{equation*}
  and consequently,
  if $k,\ell \in \NN_0$ fulfill $k+\ell = t$, then application of $\ad^{k,\ell}$ shows that $\hat s_{k,\ell} \otimes 1 \in \mathcal S$.
  By induction, $\hat s_{k,\ell} \otimes 1 \in \mathcal S$ for all $k,\ell \in \NN_0$.
\end{proof}

\begin{proposition} \label{proposition:WeylTotallySimple}
  Let $\A$ be a $^*$\=/algebra and $r\in \NN_0$.
  \begin{enumerate}
    \item \label{item:WeylTotallySimple:ideals}
      If $I$ is an ideal of $\A \otimes \Weyl{r}$, then there is an ideal $I_\A$ of $\A$ such that $I = I_\A \otimes \Weyl{r}$.
    \item \label{item:WeylTotallySimple:fingenSubalgebra}
      If $\F$ is a finitely generated $^*$\=/subalgebra of $\A \otimes \Weyl{r}$ such that
      $1 \otimes \Weyl{r} \subseteq \F$, then
      \begin{equation}
        \A_\fin \coloneqq \set{a\in \A}{a \otimes 1 \in \F}
      \end{equation}
      is a finitely generated $^*$\=/subalgebra of $\A$ and $\F = \A_\fin \otimes \Weyl{r}$.
  \end{enumerate}
\end{proposition}
\begin{proof}
  In both parts it is sufficient to consider the case $r=1$ only,
  the general case then follows by induction over $r \in \NN$ because $\Weyl{r} \otimes \Weyl{1} \cong \Weyl{r+1}$.

  Part~\ref{item:WeylTotallySimple:ideals}: Let $I$ be an ideal of $\A \otimes \Weyl{r}$ and set
  $I_\A \coloneqq \set{a\in \A}{a\otimes 1 \in I}$. Then clearly $I_\A \otimes \Weyl{r} \subseteq I$, and conversely,
  $I_\A \otimes \Weyl{r} \supseteq I$ by Lemma~\ref{lemma:tensorWeyl}.

  Part~\ref{item:WeylTotallySimple:fingenSubalgebra}: Let $\F$ be a $^*$\=/subalgebra of $\A \otimes \Weyl{r}$
  such that $1 \otimes \Weyl{r} \subseteq \F$ and assume that $\F$ is generated (as a $^*$\=/algebra) by
  $g_1,\dots,g_t \in \F$ with $t\in \NN_0$. By Lemma~\ref{lemma:tensorWeyl},
  $g_s = \sum_{k,\ell=0}^\infty g_{s,k,\ell}\otimes p^k q^\ell$ with coefficients $g_{s,k,\ell} \in \A$ such that
  $g_{s,k,\ell} \otimes 1 \in \F$ for all $s\in\{1,\dots,t\}$ and $k,\ell\in\NN_0$.
  Let $\B$ be the $^*$\=/subalgebra of $\A$ that is generated by $\set[\big]{g_{s,k,\ell}}{s\in\{1,\dots,t\};~k,\ell\in\NN_0}$
  and note that $\B$ actually is finitely generated.
  Clearly $\B \otimes \Weyl{r} \subseteq \F$ because $1\otimes \Weyl{r} \subseteq \F$.
  Conversely, $\B \otimes \Weyl{r} \supseteq \F$ because $g_s \in \B \otimes \Weyl{r}$ for all $s\in \{1,\dots,t\}$.
  So $\B \otimes \Weyl{r} = \F$.
  It is then clear that $\B \subseteq \A_\fin$, so it only remains to show that $\B \supseteq \A_\fin$.
  If $a\in \A$ fulfills $a\otimes 1 \in \F$, then $a \otimes 1 \in \B \otimes \Weyl{r}$ so that
  $a \otimes 1 = \sum_{k,\ell = 0}^\infty b_{k,\ell} \otimes p^k q^\ell$ with uniquely determined coefficients
  $b_{k,\ell} \in \B$. In particular, $b_{k,\ell} \neq 0$ only for $k=\ell =0$.
\end{proof}
Part~\ref{item:WeylTotallySimple:ideals} of Proposition~\ref{proposition:WeylTotallySimple} seems to be known, e.g.,
\cite[Lemma~4.5.1]{dixmier:envelopingAlgebras}.

By a \emph{$^*$\=/field extension of $\CC$} we mean a commutative $^*$\=/algebra that is a field,
i.e., a commutative $^*$\=/algebra such that $0\neq 1$ and every nonzero element has a multiplicative inverse.
Let $\KK$ be such a $^*$\=/field extension of $\CC$, then a \emph{$^*$\=/algebra over $\KK$} is a $^*$\=/algebra $\A$
together with a $^*$\=/algebra morphism $\zeta_\A \colon \KK \to \Zentrum(\A)$.
This way, $\A$ is also a vector space over the field $\KK$, and
we will usually suppress $\zeta_\A$ and simply write $z a \coloneqq \zeta_\A(z) a$ for $z\in \KK$, $a\in \A$.
A \emph{$\KK$\=/central $^*$\=/algebra} is a $^*$\=/algebra $\A$ over $\KK$ such that $\Zentrum(\A)$ has dimension $1$
over $\KK$, i.e., $\zeta_\A$ is an isomorphism; in this case we can simply identify $\Zentrum(\A)$ and $\KK$.
If $\A$ is a $^*$\=/algebra over some $^*$\=/field extension $\KK$ of $\CC$, and $\B$ is an arbitrary $^*$\=/algebra,
then $\A \otimes \B$ is again a $^*$\=/algebra over $\KK$, where $\zeta_{\A\otimes \B}(z) \coloneqq \zeta_\A(z) \otimes 1$
for all $z\in \KK$. If $\A$ is $\KK$\=/central and $\B$ is $\CC$\=/central, then $\A \otimes \B$ is $\KK$\=/central because
$\Zentrum(\A\otimes \B) = \Zentrum(\A) \otimes \Zentrum(\B) \cong \KK \otimes \CC \cong \KK$.

Moreover, a Lie $^*$\=/algebra over $\KK$ is a Lie algebra $\lie g$ over the field $\KK$
endowed with a $\KK$-antilinear involution $\argument^* \colon \lie g\to \lie g$ such that
$\kom{x}{y}^* = \kom{y^*}{x^*}$ for all $x,y\in \lie g$. Every $^*$\=/algebra over $\KK$ with the commutator
bracket is a Lie $^*$\=/algebra over $\KK$. We also extend the notion of enveloping $^*$\=/algebras accordingly:
An \emph{enveloping $^*$\=/algebra of a Lie $^*$\=/algebra $\lie g$ over $\KK$} is a tuple $(\A,\iota)$ of a
$^*$\=/algebra $\A$ over $\KK$ and a $\KK$-linear morphism of Lie $^*$\=/algebras $\iota \colon \lie g \to \A$
such that $\A$ is generated (as a $^*$\=/algebra over $\KK$) by $\iota(\lie g)$,
i.e., every element of $\A$ can be expressed as a $\KK$\=/linear combination of arbitrary products
(including the empty product $1$) of $\iota(\lie g)$.
Note that an enveloping $^*$\=/algebra over $\KK$ is not necessarily an enveloping $^*$\=/algebra (over $\CC$),
because the scalars need not be generated (over $\CC$) by the Lie $^*$\=/algebra and $1$.
A \emph{$\KK$-central enveloping $^*$\=/algebra} of $\lie g$ is an enveloping $^*$\=/algebra $(\A,\iota)$ of $\lie g$
such that $\A$ is $\KK$-central.
For example, let $r\in \NN_0$, then $\bigl(\KK \otimes \Weyl{r}, {\id_\KK} \otimes \iotaHeisenberg\bigr)$,
with $\id_\KK$ the identity map,
is a $\KK$-central enveloping $^*$\=/algebra of the Lie $^*$\=/algebra $\KK \otimes \Heisenberg{r}$ over $\KK$.

The main result of this section is mostly an adaptation of a classical theorem for nilpotent Lie algebras
\cite[Thm.~4.7.9]{dixmier:envelopingAlgebras} to the setting of Lie $^*$\=/algebras:

\begin{theorem} \label{theorem:weyl}
  Let $\KK$ be a $^*$\=/field extension of $\CC$,
  $\lie g$ a nilpotent Lie $^*$\=/algebra over $\KK$ such that $\lie g$ has finite dimension over $\KK$,
  and $(\A, \iota)$ a $\KK$-central enveloping $^*$\=/algebra of $\lie g$.
  Then there are $r\in \NN_0$ and a $\KK$-linear $^*$\=/algebra morphism
  $\Lambda \colon \A \to \KK \otimes \Weyl{r}$ such that $\Lambda$ is an isomorphism
  and $\Lambda\bigl(\iota(\lie g)\bigr) \subseteq \KK \otimes \WeylNil_r$.
\end{theorem}
The proof will occupy the rest of this section, throughout which we keep the notation and assumptions
of Theorem~\ref{theorem:weyl}.

First consider the case that $\A$ is commutative:
Then $\A = \Zentrum(\A) = \KK$ and we can choose $r\coloneqq 0$ so that $\Weyl{0} = \Zentrum(\Weyl{0}) = \CC = \WeylNil_0$
and $\Lambda \colon \A \to \KK \otimes \Weyl{0}$, $a \mapsto a \otimes 1$.
It is clear that this way, $\Lambda$ is an isomorphism and
$\Lambda\bigl(\iota(\lie g)\bigr) \subseteq \KK \otimes \WeylNil_0$.

Otherwise, i.e., if $\A$ is not commutative, the proof goes by induction over the dimension
of $\lie g$. For sufficiently small dimensions, namely, if $\dim_\KK (\lie g) \le 2$, the nilpotent Lie $^*$\=/algebra $\lie g$
must be abelian, which implies that $\A$ is commutative; in these cases we already know that Theorem~\ref{theorem:weyl}
holds. We make the induction assumption that Theorem~\ref{theorem:weyl} holds for all Lie $^*$\=/algebras
$\lie g'$ over $\KK$ with $\dim_\KK (\lie g') < \dim_\KK (\lie g)$. This means, if $(\A', \iota')$ is a $\KK$-central enveloping
$^*$\=/algebra of $\lie g'$, then there are $r'\in \NN_0$ and a $\KK$-linear $^*$\=/algebra morphism
$\Lambda' \colon \A' \to \KK \otimes \Weyl{r'}$ such that
$\Lambda'$ is an isomorphism and $\Lambda'\bigl(\iota'(\lie g')\bigr) \subseteq \KK \otimes \WeylNil_{r'}$.
We keep this assumption and the assumption that $\A$ is not commutative until the proof of Theorem~\ref{theorem:weyl}
is complete, i.e., until the end of this section. All the following definitions and intermediary results will be relevant
only in this section for the proof or Theorem~\ref{theorem:weyl}.

\begin{lemma}[Kirillov] \label{lemma:Kirillov}
  There are $x,y\in \lie g_\hermitian$ such that $\iota\bigl(\kom{x}{y}\bigr) = -\I$
  and $\lie g$ decomposes as $\lie g = \lie g' \oplus \KK x$ with
  $\lie g' \coloneqq \set[\big]{u \in \lie g}{\iota\bigl(\kom{u}{y}\bigr) = 0}$.
  Moreover, $\lie g'$ is a $\KK$-linear subspace and a Lie $^*$\=/subalgebra of $\lie g$,
  so $\lie g'$ is a Lie $^*$\=/algebra over $\KK$ itself. Finally, $\ad_x u \in \lie g'$ for all $u \in \lie g'$.
\end{lemma}
\begin{proof}
  This is a variant of Kirillov's Lemma for nilpotent Lie algebras with one-dimensional center,
  adapted to the setting of Lie $^*$\=/algebras. We write $\RE, \IM \colon \lie g \to \lie g_\hermitian$ for
  the hermitian and antihermitian part, i.e., $\RE(u) \coloneqq \frac{1}{2}(u+u^*)$ and
  $\IM(u) \coloneqq \frac{1}{2\I}(u-u^*)$ for all $u\in\lie g$; then $u = \RE(u)+\I \IM(u)$.

  Assume $w\in \lie g$ fulfills $\iota(w) \notin \Zentrum(\A)$. By nilpotency of $\lie g$ there is $n\in \NN$
  such that $\ad_{v_n} \cdots \ad_{v_1} w = 0$
  for all $v_1,\dots,v_n \in \lie g$, so certainly $\iota\bigl( \ad_{v_n} \cdots \ad_{v_1} w \bigr) \in \Zentrum(\A)$
  for all $v_1,\dots,v_n \in \lie g$. Hence there also is a minimal such $n\in \NN$, so
  there are $v_1,\dots,v_{n-1} \in \lie g$ such that $\iota\bigl( \ad_{v_{n-1}} \cdots \ad_{v_1} w \bigr) \notin \Zentrum(\A)$
  and $\iota\bigl( \ad_{u} \ad_{v_{n-1}} \cdots \ad_{v_1} w \bigr) \in \Zentrum(\A)$ for all $u\in \lie g$.
  Set $\tilde y \coloneqq \ad_{v_{n-1}} \cdots \ad_{v_1} w \in \lie g$. As $\iota(\tilde y) \notin \Zentrum(\A)$
  there is $y \in \bigl\{ \RE(\tilde y), \IM(\tilde y) \bigr\} \subseteq \lie g_\hermitian$ such that $\iota(y) \notin \Zentrum(\A)$.
  From $\iota\bigl(\kom{u}{\tilde y}\bigr) \in \Zentrum(\A)$ for all $u\in \lie g$
  it follows that also $\iota\bigl(\kom{u}{\tilde y^*}\bigr) = -\iota\bigl(\kom{u^*}{\tilde y}\bigr)^* \in \Zentrum(\A)$
  for all $u\in \lie g$, and therefore $\iota\bigl( \kom{u}{y} \bigr) \in \Zentrum(\A)$
  for all $u\in \lie g$.

  There also is $\tilde x \in \lie g$ such that $\iota\bigl( \kom{\tilde x}{y} \bigr) \neq 0$, because otherwise
  $\kom[\big]{\iota(u)}{\iota(y)} = \iota\bigl( \kom{u}{y} \bigr) = 0$
  for all $u\in \lie g$ would imply the contradiction $\iota(y) \in \Zentrum(\A)$.
  As $\Zentrum(\A) = \KK$, we can rescale $\tilde x$ and assume that $\iota\bigl( \kom{\tilde x}{y} \bigr) = -\I$.
  Note that ${\kom{\tilde x}{y}} = {\kom{ \RE(\tilde x)}{y}} + \I \kom{ \IM(\tilde x)}{y}$ with
  ${\kom{ \RE(\tilde x)}{y}}, {\kom{ \IM(\tilde x)}{y}} \in \I \lie g_\hermitian$.
  So  $\iota\bigl( \kom{\tilde x}{y} \bigr) = -\I$ means that $\iota\bigl( \kom{ \RE(\tilde x)}{y} \bigr) = -\I$
  and $\iota\bigl( \kom{ \IM(\tilde x)}{y} \bigr) = 0$. Set $x \coloneqq \RE(\tilde x) \in \lie g_\hermitian$,
  then $\iota(\kom{x}{y}) = -\I$ as claimed.

  It is clear that $\lie g' = \set[\big]{u \in \lie g}{\iota\bigl(\kom{u}{y}\bigr) = 0}$ is a $\KK$-linear subspace
  of $\lie g$ and a Lie $^*$\=/subalgebra, so $\lie g'$ is a Lie $^*$\=/algebra over $\KK$ itself.
  The codimension of $\lie g'$ in $\lie g$ is $1$ because the $\KK$-linear map
  $\lie g \ni u \mapsto \iota\bigl(\kom{u}{y}\bigr) \in \A$ has image $\Zentrum(\A) = \KK$.
  As $x \notin \lie g'$ it follows that $\lie g = \lie g' \oplus \KK x$.

  It only remains to show that $\ad_x u \in \lie g'$ for all $u \in \lie g'$.
  This is a general and well-known consequence for Lie subalgebras of codimension $1$ of nilpotent Lie algebras:
  For any $u \in \lie g'$ there are unique elements $\lambda_u \in \KK$ and $u' \in \lie g'$ such that
  ${\kom{u}{x}} = \lambda_u x + u'$. If $\lambda_u \neq 0$ for some $u \in \lie g'$, then by rescaling $u$ we can
  assume that ${\kom{u}{x}} - x = u' \in \lie g'$. But by induction over $n$ it would then follow that
  $\ad_u^n x - x = {\kom[\big]{u}{\ad_u^{n-1} x -x}} + {\kom{u}{x}} - x \in \lie g'$ for all $n\in \NN$,
  in particular $\ad_u^n x \neq 0$ for all $n\in \NN$, contradicting nilpotency of $\lie g$.
\end{proof}

The crucial formula for $\iota'$ in the following \eqref{eq:iotaPrimeDef} is taken from
\cite[Lemma~4.7.8]{dixmier:envelopingAlgebras}:

\begin{lemma} \label{lemma:weyldecomp}
  Let $x,y\in \lie g_\hermitian$ and $\lie g' \subseteq \lie g$ be like in Lemma~\ref{lemma:Kirillov}.
  Then $\iota' \colon \lie g' \to \A$,
  \begin{equation}
    \label{eq:iotaPrimeDef}
    u
    \mapsto
    \iota'(u)
    \coloneqq
    \sum_{m=0}^\infty \frac{(-\I)^m}{m!} \iota(y)^m  \iota\bigl( \ad_x^m u\bigr)
  \end{equation}
  is a $\KK$\=/linear Lie $^*$\=/algebra morphism such that
  \begin{equation}
    \label{eq:weyldecomp:commuting}
    \kom[\big]{\iota'(u)}{\iota(x)} = \kom[\big]{\iota'(u)}{\iota(y)} = 0
    \quad\quad\text{for all }u\in\lie g'.
  \end{equation}
\end{lemma}
\begin{proof}
  This is straightforward to check, we give the details for convenience of the reader.
  We start with \eqref{eq:weyldecomp:commuting}, so let $u\in\lie g'$.
  From $\ad_x^m u \in \lie g'$ for all $m\in \NN_0$ it follows that
  \begin{equation*}
    \kom[\big]{\iota\bigl( \ad_x^m u\bigr)}{\iota(y)}
    =
    \iota\bigl( \kom{\ad_x^m u}{y} \bigr)
    =
    0
    \quad\quad\text{for all }m\in\NN_0,
  \end{equation*}
  and therefore $\kom[\big]{\iota'(u)}{\iota(y)} = 0$. Moreover,
  from $\kom[\big]{\iota(y)}{\iota(x)} = \iota\bigl(\kom{y}{x}\bigr) = \I$ it follows that
  $\kom[\big]{\iota(y)^m}{\iota(x)} = \I m \,\iota(y)^{m-1}$ for all $m\in \NN$, and of course
  $\kom[\big]{\iota(y)^m}{\iota(x)} = 0$ in the special case $m=0$. Therefore
  \begin{align*}
    \kom[\big]{\iota'(u)}{\iota(x)}
    &=
    \sum_{m=0}^\infty \frac{(-\I)^m}{m!} \kom[\big]{\iota(y)^m  \iota\bigl( \ad_x^m u\bigr)}{\iota(x)}
    \\
    &=
    \sum_{m=1}^\infty \frac{(-\I)^{m-1}}{(m-1)!} \iota(y)^{m-1} \iota\bigl( \ad_x^m u\bigr)
    -
    \sum_{m=0}^\infty \frac{(-\I)^m}{m!} \iota(y)^m \iota\bigl(\ad_x^{m+1} u\bigr)
    \\
    &=
    0
    .
  \end{align*}
  It is clear that $\iota'$ is $\KK$\=/linear, so it only
  remains to show that $\iota'$ is compatible with the Lie bracket and the $^*$\=/involution:
  for all $u,w\in \lie g'$, the identities
  \begin{align*}
    \iota'\bigl(\kom{u}{w}\bigr)
    &=
    \sum_{m=0}^\infty \frac{(-\I)^m}{m!}
    \iota(y)^m \iota\bigl( \ad_x^m \kom{u}{w} \bigr)
    \\
    &=
    \sum_{k,\ell=0}^\infty \frac{(-\I)^{k+\ell}}{k!\,\ell!}
    \iota(y)^{k+\ell} \iota\bigl( \kom{\ad_x^k u}{\ad_x^\ell w}\bigr)
    \\
    &=
    \sum_{k,\ell=0}^\infty \frac{(-\I)^{k+\ell}}{k!\,\ell!}
    {\kom[\big]{
      \iota(y)^k \iota\bigl( \ad_x^k u \bigr)
    }{
      \iota(y)^\ell \iota\bigl( \ad_x^\ell w\bigr)
    }}
    \\
    &=
    \kom[\big]{\iota'(u)}{\iota'(w)}
  \end{align*}
  and
  \begin{equation*}
    \iota'(u^*)
    =
    \sum_{m=0}^\infty \frac{(-\I)^m}{m!} \iota(y)^m \iota\bigl( \ad_x^m u^*\bigr)
    =
    \sum_{m=0}^\infty \frac{(-\I)^m}{m!} (-1)^m \iota(y)^m \iota\bigl( \ad_x^m u\bigr)^*
    =
    \iota'(u)^*
  \end{equation*}
  hold.
\end{proof}
Note that $\iota'$ is not injective, e.g., $\iota'(y) = 0$. We introduce some more notation.

\begin{definition}
  We write $\A'$ for the $^*$\=/subalgebra over $\KK$ of $\A$ that is generated by $\iota'(\lie g')$,
  i.e., $\A'$ is the set of $\KK$\=/linear combinations of arbitrary products (including the empty product $1$)
  of elements of $\iota'(\lie g')$.
\end{definition}
Note that $(\A',\iota')$ is an enveloping $^*$\=/algebra of the Lie $^*$\=/algebra $\lie g'$ over $\KK$.
Yet before we can apply the induction hypothesis to $(\A',\iota')$, we have to show that $\A'$ is $\KK$\=/central.

Recall that the elements $x,y\in \lie g_\hermitian$ from Lemma~\ref{lemma:Kirillov} fulfill the canoncial
commutation relations $\kom{\iota(x)}{\iota(y)} = \iota\bigl(\kom{x}{y}\bigr) = -\I$.

\begin{definition}
  For simplicity, we write $p\coloneqq p_1 \in \Weyl{1}$ and $q \coloneqq q_1 \in \Weyl{1}$ for the
  generators of $\Weyl{1}$. Let $\Xi \colon \Weyl{1} \to \A$ be the $^*$\=/algebra morphism such that
  $\Xi(p) = \iota(x)$ and $\Xi(q) = \iota(y)$. Then we define the $\CC$\=/linear map
  $\Psi \colon \A' \otimes \Weyl{1} \to \A$,
  \begin{equation}
    a \otimes b \mapsto \Psi(a \otimes b) \coloneqq a\,\Xi(b)
    .
  \end{equation}
\end{definition}

\begin{proposition} \label{proposition:Psi}
  The $\CC$\=/linear map $\Psi \colon \A' \otimes \Weyl{1} \to \A$ has the following properties:
  \begin{enumerate}
    \item \label{item:Psi:Klinmorph} $\Psi$ is a $\KK$\=/linear $^*$\=/algebra morphism.
    \item \label{item:Psi:iso} $\Psi$ is an isomorphism.
    \item \label{item:Psi:ongen} $\Psi^{-1}\bigl(\iota(\lie g) \bigr) \subseteq \KK \otimes p + \sum_{n=0}^\infty \iota'(\lie g') \otimes q^n$.
  \end{enumerate}
\end{proposition}
\begin{proof}
  First we check that $\Psi$ is a $\KK$\=/linear $^*$\=/algebra morphism.
  By definition, $\Psi\colon \A' \otimes \Weyl{1} \to \A$ is the $\CC$\=/linear map corresponding to the $\CC$\=/bilinear map
  $\A' \times \Weyl{1} \ni (a,b) \mapsto a \,\Xi(b) \in \A$ by the universal property of the tensor product.
  However, as $\Psi$ is the identity on the first
  factor, it is also $\KK$\=/linear, because $\Psi\bigl(z(a\otimes b)\bigr) = \Psi(za\otimes b) = za \, \Xi(b)
  = z \Psi(a\otimes b)$ for all $z\in \KK$ and $a \in \A'$, $b\in \Weyl{1}$.
  As the $^*$\=/subalgebras $\A'$ and $\Xi(\Weyl{1})$ of $\A$ are generated by $\iota'(\lie g')$ and $\{\iota(x),\iota(y)\}$,
  respectively, it follows from \eqref{eq:weyldecomp:commuting} that $\kom{a}{\Xi(b)} = 0$ for all $a \in \A'$, $b\in \Weyl{1}$.
  This guarantees that $\Psi$ is a $^*$\=/algebra morphism, which completes the proof of part~\ref{item:Psi:Klinmorph}.

  Next we prove parts~\ref{item:Psi:iso} and \ref{item:Psi:ongen}. Write $I \coloneqq \ker(\Psi)$ for the kernel of $\Psi$.
  By part~\ref{item:WeylTotallySimple:ideals} of Proposition~\ref{proposition:WeylTotallySimple},
  there is an ideal $I_{\A'}$ of $\A'$ such that $\ker(\Psi) = I_{\A'} \otimes \Weyl{1}$.
  But for any $a\in I_{\A'}$ it then follows that $a = a\,\Xi(1) = \Psi(a\otimes 1) = 0$, so $I_{\A'} = \{0\}$
  and consequently $\ker(\Psi) = \{0\}$, which means that $\Psi$ is injective.

  Now recall that $\lie g = \lie g' \oplus \KK x$ and that $\ad_x$ restricts to a derivation of $\lie g'$
  (see Lemma~\ref{lemma:Kirillov}).
  In order to prove surjectivity of $\Psi$ and part~\ref{item:Psi:ongen}, we construct the
  $\KK$\=/linear map $\phi \colon \lie g \to \A' \otimes \Weyl{1}$,
  \begin{equation*}
    \begin{split}
      x &\mapsto \phi(x) \coloneqq 1 \otimes p \\
      u &\mapsto \phi(u) \coloneqq \sum_{n=0}^\infty \frac{\I^n}{n!} \iota'\bigl(\ad_x^n u\bigr) \otimes q^n
        \quad\quad\text{for }u \in \lie g'.
    \end{split}
  \end{equation*}
  This formula for $\phi$ is again taken from \cite[Lemma~4.7.8]{dixmier:envelopingAlgebras}.
  The proof is complete once we have shown that $\Psi \circ \phi = \iota$, because this, first,
  implies that the generating subset $\iota(\lie g)$ of $\A$ is contained in the image of $\Psi$,
  so $\Psi$ must be surjective and therefore an isomorphism. And second, it then follows that $\phi = \Psi^{-1}\circ \iota$,
  so $\Psi^{-1}\bigl(\iota(\lie g) \bigr) = \phi(\lie g)$, which clearly is a subset of
  $\KK \otimes p + \sum_{n=0}^\infty \iota'(\lie g') \otimes q^n$.

  As $\Psi$, $\phi$, and $\iota$ are $\KK$\=/linear maps, it is sufficient to check the identity $\Psi \circ \phi = \iota$
  on $x$ and on $\lie g'$. Clearly $\Psi \bigl(\phi(x)\bigr) = \Psi (1\otimes p) = \Xi(p) = \iota(x)$.
  Next let $u\in \lie g'$, then
  \begin{align*}
    \Psi \bigl(\phi(u)\bigr)
    &=
    \sum_{n=0}^\infty \frac{\I^n}{n!} \Psi \bigl( \iota'\bigl(\ad_x^n u\bigr) \otimes q^n \bigr)
    \\
    &=
    \sum_{n=0}^\infty \frac{\I^n}{n!} \iota'\bigl(\ad_x^n u\bigr) \,\iota(y)^n
    \\
    &=
    \sum_{n,m=0}^\infty \frac{\I^n\,(-\I)^m}{n!\,m!} \iota(y)^{n+m}\, \iota\bigl( \ad_x^{n+m} u\bigr)
    \\
    &=
    \sum_{k=0}^\infty \sum_{m=0}^k \binom{k}{m} \frac{\I^{k-m} (-\I)^m}{k!} \iota(y)^k \,\iota\bigl(\ad_x^k u\bigr)
    \\
    &=
    \iota(u) + \sum_{k=1}^\infty \frac{0^k}{k!} \iota( y)^k \, \iota\bigl( \ad_x^k u\bigr)
    \\
    &=
    \iota(u)
    ,
  \end{align*}
  where we use the substitution $k \coloneqq n+m$.
\end{proof}

\begin{corollary} \label{corollary:Psisim}
  The enveloping $^*$\=/algebra $(\A',\iota')$ of $\lie g'$ is $\KK$\=/central.
\end{corollary}
\begin{proof}
  Clearly $\Zentrum(\A') \otimes \Zentrum(\Weyl{1}) = \Zentrum(\A') \otimes \CC \cong \Zentrum(\A')$.
  Moreover, as $\Psi \colon \A' \otimes \Weyl{1} \to \A$ is a $\KK$\=/linear $^*$\=/algebra morphism and an isomorphism
  by Proposition~\ref{proposition:Psi},
  $\Zentrum(\A') \otimes \Zentrum(\Weyl{1}) = \Zentrum(\A' \otimes \Weyl{1}) \cong \Zentrum(\A) = \KK$.
  Therefore $\Zentrum(\A') \cong \KK$, and in particular, $\Zentrum(\A')$ has dimension $1$ over $\KK$.
\end{proof}

\begin{proof}[of Theorem~\ref{theorem:weyl}]
  Combining Corollary~\ref{corollary:Psisim} and the induction hypothesis shows that
  there are $r'\in \NN_0$ and a $\KK$\=/linear $^*$\=/algebra morphism
  $\Lambda' \colon \A' \to \KK \otimes \Weyl{r'}$
  such that $\Lambda'$ is an isomorphism and $\Lambda' \bigl(\iota'(\lie g')\bigr) \subseteq \KK \otimes \WeylNil_{r'}$.

  Set $r\coloneqq r'+1$. First we define the $^*$\=/algebra morphism
  $\Theta \colon \Weyl{r} \to \Weyl{r'} \otimes \Weyl{1}$
  \begin{align*}
    &&
    p_j &\mapsto \Theta(p_j) \coloneqq p_j \otimes 1
    &&\text{and}&
    q_j &\mapsto \Theta(q_j) \coloneqq q_j \otimes 1
    &&\text{for }j\in\{1,\dots,r'\}
  \shortintertext{and}
    &&
    p_r &\mapsto \Theta(p_r) \coloneqq 1 \otimes p
    &&\text{and}&
    q_r &\mapsto \Theta(q_r) \coloneqq 1 \otimes q
    .
    &&
  \end{align*}
  Clearly, $\Theta$ is an isomorphism.
  By Proposition~\ref{proposition:Psi}, $\Psi \colon \A' \otimes \Weyl{1} \to \A$ also is a $\KK$\=/linear $^*$\=/algebra
  morphism and an isomorphism. We define the composition
  \begin{equation*}
    \Lambda
    \coloneqq
    \bigl({\id_\KK} \otimes \Theta^{-1} \bigr) \circ \bigl( \Lambda' \otimes {\id_{\Weyl{1}}} \bigr) \circ \Psi^{-1}
    \colon
    \A \to \KK \otimes \Weyl{r}
    ,
  \end{equation*}
  i.e., $\Lambda$ makes the diagram
  \begin{equation}
    \begin{tikzcd}[sep = large, column sep=huge]
      \A \arrow[r, "\Psi^{-1}"] \arrow[d,"\Lambda",swap] & \A' \otimes \Weyl{1} \arrow[d, "\Lambda' \otimes {\id_{\Weyl{1}}}"]\\
      \KK \otimes \Weyl{r} & \KK \otimes \Weyl{r'} \otimes \Weyl{1} \arrow[l, "\quad{\id_\KK} \otimes {\Theta^{-1}}", swap]
    \end{tikzcd}
  \end{equation}
  commute. By construction, $\Lambda$ is a $\KK$\=/linear $^*$\=/algebra morphism and an isomorphism.

  Part~\ref{item:Psi:ongen} of Proposition~\ref{proposition:Psi} shows that
  $\Psi^{-1}\bigl(\iota(\lie g) \bigr) \subseteq \KK \otimes p + \sum_{n=0}^\infty \iota'(\lie g') \otimes q^n$.
  As $\Lambda'$ is a $\KK$\=/linear $^*$\=/algebra morphism and
  $\Lambda'\bigl(\iota'(\lie g')\bigr) \subseteq \KK \otimes \WeylNil_{r'}$, it follows that
  \begin{equation*}
    \bigl( \Lambda' \otimes {\id_{\Weyl{1}}} \bigr)\Bigl( \Psi^{-1}\bigl(\iota(\lie g)\bigr)\Bigr)
    \subseteq
    \KK \otimes 1\otimes p + \sum_{n=0}^\infty \KK \otimes \WeylNil_{r'} \otimes q^n
    .
  \end{equation*}
  Finally, $\Theta^{-1}(1\otimes p) = p_r \in \WeylNil_r$ and
  $\Theta^{-1}(\WeylNil_{r'} \otimes q^n) = \set[\big]{aq_r^n}{a \in \WeylNil_{r-1}} \subseteq \WeylNil_r$ for all $n\in \NN_0$,
  which follows immediately from Definition~\ref{definition:WeylNil}.
  Therefore $\Lambda\bigl(\iota(\lie g) \bigr) \subseteq \KK \otimes \WeylNil_r$.
\end{proof}

\section{Real prime ideals} \label{sec:realprime}
Consider any $^*$\=/algebra $\A$. For $I, J \subseteq \A$ we write
\begin{equation}
  I \cdot J \coloneqq \set[\big]{ab}{a\in I, b\in J}
  \quad\quad\text{and}\quad\quad
  (I:J) \coloneqq \set[\big]{a\in \A}{ a J \subseteq I}
  .
\end{equation}
Clearly $(I:J)\cdot J \subseteq I$.
If $I$ and $J$ are ideals of $\A$, then $(I:J)$ is again an ideal of $\A$, called the \emph{ideal quotient} of $I$ over $J$.
However, in our convention, $I\cdot J$ is, in general, not closed under addition, hence not an ideal.
Recall that an ideal $I$ of $\A$ is called \emph{proper} if $I \neq \A$, and
$I$ is called \emph{prime} if $I$ is proper and the following holds:
Whenever two ideals $J$ and $K$ of $\A$ fulfill $I \subseteq J$, $I \subseteq K$, and $J \cdot K \subseteq I$,
then $I=J$ or $I=K$.
Equivalently, a proper ideal $I$ of $\A$ is prime if and only if the following holds:
Whenever two elements $a,c\in \A$ fulfill $abc \in I$ for all $b\in \A$, then $a\in I$ or $c\in I$.

\begin{lemma} \label{lemma:realIdealDecomposition}
  Let $\A$ be a $^*$\=/algebra and $I, J$ two ideals of $\A$. If $I$ is real, then $(I:J)$ is also real.
\end{lemma}
\begin{proof}
  Assume $I$ is real and consider $k\in \NN$ and $a_1,\dots,a_k \in \A$ such that $\sum_{j=1}^k a_j^* a_j \in (I:J)$.
  Then $\sum_{j=1}^k a_j^* a_j c \in I$ for all $c\in J$, and therefore
  $\sum_{j=1}^k (a_j c)^* (a_j c)  = c^* \sum_{j=1}^k a_j^* a_j c \in I$ for all $c\in J$.
  It follows that $a_j c \in I$ for all $j\in \{1,\dots,k\}$ and $c\in J$, so $a_1, \dots, a_k \in (I:J)$.
\end{proof}
It is well-known that every semiprime ideal of a Noetherian ring is the intersection of finitely many prime ideals
(possibly the empty intersection, which is the whole ring by definition).
This intersection theorem can be adapted to $^*$\=/algebras in the following way:

\begin{theorem} \label{theorem:realIdealDecomposition}
  Let $\A$ be a Noetherian $^*$\=/algebra. Then every real ideal of $\A$ is the intersection of
  finitely many real prime ideals of $\A$ (possibly the empty intersection which by definition equals $\A$).
\end{theorem}
\begin{proof}
  Let $E$ be the set of all real ideals of $\A$ that are not the intersection of finitely many real prime ideals of $\A$
  (where the empty intersection qualifies as an intersection of finitely many real prime ideals).
  We want to show that $E$ is empty, so assume to the contrary that $E$ is nonempty.
  As $\A$ is Noetherian, it follows that there is a maximal (with respect to $\subseteq$) element $I$ of $E$.
  In particular, $I$ is a real proper ideal of $\A$ (and therefore also a $^*$\=/ideal), and $I$ is not prime,
  which means there are two ideals $J$ and $K$ of $\A$ such that $I \subsetneq J$, $I \subsetneq K$, and $J \cdot K \subseteq I$.
  Recall that $J^* = \set{a^*}{a\in J}$ is also an ideal and
  construct the ideal quotients $\overline{K} \coloneqq (I:J^*)$ and $\overline{J} \coloneqq (I:\overline{K})$,
  then clearly $\overline J \cdot \overline K = (I:\overline{K}) \cdot \overline K \subseteq I$.
  By Lemma~\ref{lemma:realIdealDecomposition}, both $\overline{J}$ and $\overline{K}$ are real ideals of $\A$,
  and in particular, $^*$\=/ideals. Now consider any $a \in \overline J \cap \overline K$, then
  $a^* a \in \overline J^* \cdot \overline K = \overline J \cdot \overline K \subseteq I$,
  and therefore $a\in I$ because $I$ is real. So $\overline J \cap \overline K \subseteq I$.
  Next we show that $I \subsetneq \overline J$ and $I \subsetneq \overline K$:
  If $j\in J$ and $k\in \overline K$, then $k^* \in \overline K^* = \overline K = (I:J^*)$, and therefore
  $jk = (k^* j^*)^* \in I^* = I$; this shows that $J \subseteq \overline{J}$, and therefore $I \subsetneq \overline J$.
  Similarly, if $k \in K$ and $j\in J$, then $k^* j^* = (jk)^* \in I^* = I$; this shows that $K^* \subseteq \overline{K}$,
  hence $K \subseteq \overline{K}^* = \overline K$, and therefore $I \subsetneq \overline K$.
  To summarize, both $\overline J$ and $\overline K$ are real ideals of $\A$ such that
  $I \subsetneq \overline J$, $I \subsetneq \overline K$, and $\overline J \cap \overline K \subseteq I$.
  This first of all implies that $\overline J \cap \overline K = I$ and that $\overline J \neq \A$ and $\overline K \neq \A$.
  By maximality of $I$ also $\overline J \notin E$ and $\overline K \notin E$ hold,
  and as we already know that both $\overline J$ and $\overline K$ are real ideals of $\A$,
  this means that both $\overline J$ and $\overline K$ are intersections of finitely many real prime ideals of $\A$.
  Therefore their intersection $I$ also is an intersection of finitely many real prime ideals of $\A$,
  in contradiction to $I \in E$.
\end{proof}

\begin{corollary} \label{corollary:realIdealDecomposition}
  Let $\lie g$ be a nilpotent finite-dimensional Lie $^*$\=/algebra and $(\A,\iota)$ an enveloping $^*$\=/algebra of $\lie g$.
  Then every real ideal of $\A$ is the intersection of finitely many real prime ideals of $\A$
  (possibly the empty intersection which by definition equals $\A$).
\end{corollary}
\begin{proof}
  By Remark~\ref{remark:envelopingquotient}, $\A$ is Noetherian, so Theorem~\ref{theorem:realIdealDecomposition} applies.
\end{proof}

Corollary~\ref{corollary:realIdealDecomposition} allows us to reduce the real Nullstellensatz to the case of prime enveloping
$^*$\=/algebras, which can be localized at nonzero central elements:

Consider a prime $^*$\=/algebra $\P$, i.e., $\P$ is a $^*$\=/algebra such that $\{0\}$ is a prime ideal of $\P$.
Then $\P$ does not have central zero divisors
(if $ac = 0$ with $a\in \P$ and $c\in \Zentrum(\P)$, then $abc = acb = 0$ for all $b\in \P$,
so $a=0$ or $c=0$ because $\{0\}$ is a prime ideal of $\P$). In particular,
$\Zentrum(\P) \setminus \{0\}$ is closed under multiplication.
The \emph{localization of $\P$ at $\Zentrum(\P) \setminus \{0\}$} is the $^*$\=/algebra $\P_{\totloc}$
of equivalence classes of $\P \times (\Zentrum(\P) \setminus \{0\})$
with respect to the equivalence relation $\sim$ that is defined as
\begin{equation}
  (a,d) \sim (a',d') \quad\quad\Longleftrightarrow\quad\quad ad' = a'd
  \quad\quad\text{for all }a\in \P \text{ and }d \in \Zentrum(\P) \setminus \{0\}.
\end{equation}
In order to check transitivity of $\sim$ one can use that $\P$ does not have central zero divisors.
We write $\frac{a}{d}$ for the $\sim$\,-equivalence class of $(a,d) \in \P \times (\Zentrum(\P) \setminus \{0\})$
and define addition, multiplication, and $^*$\=/involution of $\P_{\totloc}$ as usual:
\begin{equation}
  \frac{a}{d} + \frac{b}{e} \coloneqq \frac{ae+bd}{de},
  \quad\
  \frac{a}{d} \cdot \frac{b}{e} \coloneqq \frac{ab}{de},
  \quad\text{and}\quad
  \biggl(\frac{a}{d} \biggr)^* \coloneqq \frac{a^*}{d^*}
\end{equation}
for all $\frac{a}{d},\frac{b}{e} \in \P_{\totloc}$.
It is straightforward to check that these three operations are welldefined. Note that
two elements $a\in \P$, $d\in \Zentrum(\P) \setminus \{0\}$ fulfill $\frac{a}{d} = 0$ (i.e., $\frac{a}{d} = \frac01$)
if and only if $a = 0$. The unit of $\P_\totloc$ is of course the element $\frac{1}{1}$.

\begin{proposition} \label{proposition:localization}
  Let $\P$ be a prime $^*$\=/algebra.
  Then the center of its localization at $\Zentrum(\P) \setminus \{0\}$ is
  \begin{equation}
    \label{eq:localization:zentrum}
    \Zentrum(\P_{\totloc}) = \set[\bigg]{\frac{a}{d}}{a\in\Zentrum(\P), d \in \Zentrum(\P) \setminus \{0\} },
  \end{equation}
  which is a $^*$\=/field extension of $\CC$.
  If in addition $\P$ is formally real, then $\Zentrum(\P_\totloc)$ is formally real.
\end{proposition}
\begin{proof}
  It is clear that $\frac{a}{d} \in \Zentrum(\P_{\totloc})$ for all $a\in\Zentrum(\P)$, $d \in \Zentrum(\P) \setminus \{0\}$.
  Conversely, assume that $a\in \P$, $d\in \Zentrum(\P) \setminus \{0\}$ fulfill $\frac{a}{d} \in \Zentrum(\P_{\totloc})$.
  Then for all $b\in \P$ the identity $\frac{ab-ba}{d} = \frac{a}{d}\frac{b}{1} - \frac{b}{1} \frac{a}{d} = 0$ holds,
  and therefore $ab-ba = 0$. This shows that $a\in\Zentrum(\P)$, which then completes the proof of \eqref{eq:localization:zentrum}.
  It follows from \eqref{eq:localization:zentrum} that every nonzero element $\frac ad \in \Zentrum(\P_{\totloc})$ has a
  multiplicative inverse $\frac da\in \Zentrum(\P_{\totloc})$, and clearly $\frac{0}{1} \neq \frac{1}{1}$ because
  $\P$ is prime and therefore $\P \neq \{0\}$ (recall that the prime ideal $\{0\}$ of $\P$ is a proper ideal by definition).

  Now assume that $\P$ is formally real and let $k\in \NN$ and $a_1,\dots,a_k \in \Zentrum(\P)$, $d_1,\dots, d_k \in \Zentrum(\P) \setminus \{0\}$
  such that $\sum_{j=1}^k \bigl(\frac{a_j}{d_j}\bigr)^{\!*\,} \frac{a_j}{d_j} = 0$. By multiplication with $\prod_{j=1}^k d_j^* d_j$
  it follows that $\sum_{j=1}^k b_j^* b_j = 0$ where
  $b_j \coloneqq a_j \prod_{i \in \{1,\dots,k\}\setminus\{j\}} d_i \in \Zentrum(\P)$,
  so $b_1 = \dots = b_k = 0$ because $\P$ is formally real. As $\P$ does not have central zero divisors
  and $d_1,\dots,d_k \in \Zentrum(\P) \setminus \{0\}$, it follows that
  $a_1 = \dots = a_k = 0$. This shows that $\Zentrum(\P_{\totloc})$ is formally real.
\end{proof}
Consider a prime $^*$\=/algebra $\P$ and $z \in \Zentrum(\P)_\hermitian \setminus \{0\}$.
Then it is easy to check that
\begin{equation}
  \P_{z\mloc} \coloneqq \set[\bigg]{\frac{a}{z^k}}{a \in \P, k \in \NN_0}
\end{equation}
is a $^*$\=/subalgebra of $\P_\totloc$.

\begin{proposition} \label{proposition:zlocalization}
  Let $\P$ be a prime $^*$\=/algebra and let $k\in \NN_0$ and $f_1,\dots, f_k \in \P_\totloc$.
  Then there are $z \in \Zentrum(\P)_\hermitian \setminus \{0\}$ and $a_1,\dots,a_k \in \P$ such that
  \begin{equation}
    f_j = \frac{a_j}{z} \in \P_{z\mloc}
    \quad\quad\text{for all }j\in\{1,\dots,k\}
    .
  \end{equation}
\end{proposition}
\begin{proof}
  There are $b_j \in \P$ and $d_j \in \Zentrum(\P)\setminus \{0\}$ such that $f_j = \frac{b_j}{d_j}$
  for all $j\in \{1,\dots,k\}$. Set $z \coloneqq \prod_{j=1}^k d_j^*d_j$, then $z \in \Zentrum(\P)_\hermitian\setminus \{0\}$
  because $\P$ does not have central zero divisors. For $j\in\{1,\dots,k\}$, set
  $a_j \coloneqq d_j^* b_j \prod_{i\in\{1,\dots,k\}\setminus\{j\}} d_i^*d_i$, then
  \smash{$\frac{a_j}{z} = \frac{b_j}{d_j} = f_j$}.
\end{proof}

\section{Proof of the main theorem} \label{sec:proof}
We can now combine the results of Sections~\ref{sec:weyl}, \ref{sec:centralEnv}, and \ref{sec:realprime}
and prove the real Nullstellensatz for finite-dimensional nilpotent Lie $^*$\=/algebras.
The following Lemma~\ref{lemma:mainLemma} contains the main technical steps of the proof:

\begin{lemma} \label{lemma:mainLemma}
  Let $\lie g$ be a nilpotent finite-dimensional Lie $^*$\=/algebra, $(\P,\kappa)$ an enveloping $^*$\=/algebra of $\lie g$,
  and $\hat a \in \P$. Assume that the $^*$\=/algebra $\P$ is prime and formally real and that $\hat a \neq 0$.
  Then there are $r\in \NN_0$ and a surjective $^*$\=/algebra morphism
  $\Phi \colon \P \to \Weyl{r}$ such that $\Phi(\hat a) \neq 0$ and $\Phi\bigl( \kappa(\lie g) \bigr) \subseteq \WeylNil_r$.
\end{lemma}
\begin{proof}
  The following commutative diagram gives an overview over the most important $^*$\=/algebra morphisms that will occur in the
  construction of $\Phi$ (hooked arrows ``$\hookrightarrow$'' denote the canonical inclusion of a $^*$\=/subalgebra):
  \begin{equation*}
    \begin{tikzcd}[sep = large, column sep=huge]
      \P \arrow[r, "a\mapsto \frac{a}{1}"] \arrow[d,"\Phi"]
      & \P_{z\mloc} \arrow[r, hook] \arrow[d, "\Lambda\at{\P_{z\mloc}}"] & \P_{\totloc}  \arrow[d,"\Lambda", swap] \\
      \Weyl{r} & \KK_\fin \otimes \Weyl{r} \arrow[l, "\widehat\varphi", swap] \arrow[r, hook] & \KK \otimes \Weyl{r}
    \end{tikzcd}
  \end{equation*}
  Proposition~\ref{proposition:localization} shows that the center $\Zentrum(\P_\totloc)$ of the localization $\P_\totloc$
  of $\P$ is a formally real $^*$\=/field extension of $\CC$;
  we write $\KK \coloneqq \Zentrum(\P_\totloc)$. So $\P_\totloc$ is $\KK$\=/central.
  We then define the $\CC$\=/linear map $\iota \colon \KK \otimes \lie g \to \P_\totloc$,
  \begin{equation*}
    c \otimes x \mapsto \iota(c\otimes x) \coloneqq c\,\frac{\kappa(x)}{1}
    .
  \end{equation*}
  Clearly, $\iota$ is a $\KK$-linear Lie $^*$\=/algebra morphism
  and $\iota(\KK\otimes \lie g)$ generates $\P_\totloc$ as a $^*$\=/algebra, so $(\P_\totloc,\iota)$ is a $\KK$-central
  enveloping $^*$\=/algebra of $\KK\otimes \lie g$. By Theorem~\ref{theorem:weyl}, there are $r\in \NN_0$
  and a $\KK$\=/linear $^*$\=/algebra morphism $\Lambda \colon \P_\totloc \to \KK \otimes \Weyl{r}$
  such that $\Lambda$ is an isomorphism and $\Lambda\bigl(\iota(\KK\otimes \lie g)\bigr) \subseteq \KK \otimes \WeylNil_r$.
  In particular, $\Lambda\bigl(\frac{\kappa(x)}{1} \bigr) = \Lambda\bigl(\iota(1\otimes x)\bigr) \in \KK\otimes \WeylNil_r$
  for all $x\in\lie g$.

  Like in Section~\ref{sec:weyl}, we write $p_1,\dots,p_r,q_1,\dots,q_r$ for the canoncial generators of $\Weyl{r}$.
  Moreover, let $d \coloneqq \dim_\CC(\lie g)$ and let $b_1,\dots,b_d$ be a basis of $\lie g$.
  Then we can expand $\Lambda\bigl( \frac{\kappa(b_m)}{1} \bigr) \in \KK \otimes \WeylNil_r$ as
  \begin{equation*}
    \Lambda\biggl( \frac{\kappa(b_m)}{1} \biggr) = \sum_{n=1}^\ell f_{m,n} \otimes w_{m,n}
    \quad\quad\text{for all }m\in\{1,\dots,d\}
  \end{equation*}
  with suitable $\ell \in \NN$ and $f_{m,n} \in \KK$, $w_{m,n} \in \WeylNil_r$ for $m \in \{1,\dots,d\}$, $n\in \{1,\dots,\ell\}$.
  By Proposition~\ref{proposition:zlocalization}, there are $z \in \Zentrum(\P)_\hermitian \setminus \{0\}$ and
  $p_1', \dots, p_r', q_1', \dots, q_r' \in \P$ and $g_{m,n} \in \P$ for $m \in \{1,\dots,d\}$, $n\in \{1,\dots,\ell\}$ such that
  \begin{equation*}
    \tag{\theoremMainInternalTag}
    \label{eq:internal:pqPrimeDef}
    \Lambda^{-1}(1\otimes p_j) = \frac{p'_j}{z} \in \P_{z\mloc}
    \quad\quad\text{and}\quad\quad
    \Lambda^{-1}(1\otimes q_j) = \frac{q'_j}{z} \in \P_{z\mloc}
    \quad\quad\text{for all $j\in\{1,\dots,r\}$}
  \end{equation*}
  and
  \begin{equation*}
    f_{m,n} = \frac{g_{m,n}}{z} \in \P_{z\mloc}
    \quad\quad\text{for all $m \in \{1,\dots,d\}$ and $n\in \{1,\dots,\ell\}$.}
  \end{equation*}
  The $^*$\=/subalgebra $\P_{z\mloc}$ of $\P_\totloc$ is generated by
  $\frac{\kappa(b_1)}{1}, \dots, \frac{\kappa(b_d)}{1}, \frac{1}{z}$.
  Therefore, $\Lambda( \P_{z\mloc} )$ is a finitely generated $^*$\=/subalgebra of $\KK \otimes \Weyl{r}$.
  By construction of $z$ we have $1\otimes p_j = \Lambda\bigl(\frac{p'_j}{z}\bigr) \in \Lambda(\P_{z\mloc})$
  and $1\otimes q_j = \Lambda\bigl(\frac{q'_j}{z}\bigr) \in \Lambda(\P_{z\mloc})$
  for all $j\in\{1,\dots,r\}$, so $1 \otimes \Weyl{r} \subseteq \Lambda(\P_{z\mloc})$.
  Part~\ref{item:WeylTotallySimple:fingenSubalgebra} of Proposition~\ref{proposition:WeylTotallySimple} now shows that
  $\KK_\fin \coloneqq \set[\big]{c\in \KK}{c \otimes 1 \in \Lambda( \P_{z\mloc} )}$
  is a finitely generated $^*$\=/subalgebra of $\KK$ and that $\Lambda( \P_{z\mloc} ) = \KK_\fin \otimes \Weyl{r}$.
  As $\Lambda$ is a $\KK$\=/linear $^*$\=/algebra morphism,
  \begin{equation*}
    \Lambda(c)
    =
    \Lambda(c\cdot \mathbf{1}_{\P_\totloc})
    =
    c\,\Lambda(\mathbf{1}_{\P_\totloc})
    =
    c\,\bigl(\mathbf{1}_{\KK} \otimes \mathbf{1}_{\Weyl{r}})
    =
    c \otimes \mathbf{1}_{\Weyl{r}}
    \quad\quad\text{for all }c\in \Zentrum(\P_\totloc) = \KK,
  \end{equation*}
  where $\mathbf{1}_{\P_\totloc}$, $\mathbf{1}_{\KK}$, and $\mathbf{1}_{\Weyl{r}}$ are the units of the $^*$\=/algebras
  $\P_\totloc$, $\KK$, and $\Weyl{r}$, respectively. As $\Lambda$ is an isomorphism, it follows that an element
  $c\in \KK$ fulfills $c \otimes 1 \in \Lambda( \P_{z\mloc} )$ if and only if $c\in \P_{z\mloc}$. Therefore $\KK_\fin = \KK \cap \P_{z\mloc}$.
  In particular, $f_{m,n} \in \KK_\fin$ for all $m \in \{1,\dots,d\}$ and $n\in \{1,\dots,\ell\}$, so
  \begin{equation*}
    \Lambda\biggl( \frac{\kappa(b_m)}{1} \biggr) = \sum_{n=1}^\ell f_{m,n} \otimes w_{m,n} \in \KK_\fin \otimes \WeylNil_r
    \quad\quad\text{for all $m\in\{1,\dots,d\}$,}
  \end{equation*}
  and consequently $\Lambda\bigl( \frac{\kappa(x)}{1} \bigr) \in \KK_\fin \otimes \WeylNil_r$ for all $x\in\lie g$.

  As $\KK$ is formally real,
  its finitely generated $^*$\=/subalgebra $\KK_\fin$ is also formally real. The commutative real Nullstellensatz
  in the form of Theorem~\ref{theorem:NullstellensatzCommutative} now shows that
  $\bigcap_{\varphi \in \starchar(\KK_\fin)} \ker(\varphi) = \{0\}$, where $\starchar(\KK_\fin)$ is the set of $^*$\=/algebra morphisms
  from $\KK_\fin$ to $\CC$. For every such $\varphi \in \starchar(\KK_\fin)$,
  we define the $\CC$\=/linear map $\widehat\varphi \colon \KK_\fin \otimes \Weyl{r} \to \Weyl{r}$,
  \begin{equation*}
    c \otimes a \mapsto \widehat\varphi(c\otimes a) \coloneqq \varphi(c)\, a
    .
  \end{equation*}
  Then $\widehat\varphi$ is a $^*$\=/algebra morphism because $\varphi(c)\,a = a \,\varphi(c)$ for all $c\in \KK_\fin$, $a\in \Weyl{r}$.

  Consider the $^*$\=/ideal $I \coloneqq \bigcap_{\varphi\in\starchar(\KK_\fin)} \ker(\widehat\varphi)$ of $\KK_\fin \otimes \Weyl{r}$.
  By part~\ref{item:WeylTotallySimple:ideals} of Proposition~\ref{proposition:WeylTotallySimple},
  there is an ideal $I_{\KK_\fin}$ of $\KK_\fin$ such that $I = I_{\KK_\fin} \otimes \Weyl{r}$.
  However, for all $c\in I_{\KK_\fin}$ and all $\varphi\in\starchar(\KK_\fin)$,
  the identity $\varphi(c) = \widehat\varphi(c\otimes 1) = 0$ holds, and therefore
  $I_{\KK_\fin} \subseteq \bigcap_{\varphi \in \starchar(\KK_\fin)} \ker(\varphi) = \{0\}$. Con\-sequently
  $\bigcap_{\varphi\in\starchar(\KK_\fin)} \ker(\widehat\varphi) = I = I_{\KK_\fin} \otimes \Weyl{r} = \{0\}$.
  Note that $\Lambda\bigl( \frac{\hat a}{1} \bigr) \neq 0$ because $\Lambda$ is an isomorphism and $\hat a\neq0$.
  Therefore, there is $\varphi \in \starchar(\KK_\fin)$ such that $\widehat\varphi\bigl( \Lambda\bigl( \frac{\hat a}{1} \bigr) \bigr) \neq 0$.
  We choose one such $^*$\=/character $\varphi$ and use it to define the $^*$\=/algebra morphism
  $\Phi \colon \P \to \Weyl{r}$,
  \begin{equation*}
    a \mapsto \Phi(a) \coloneqq \widehat\varphi\biggl( \Lambda\biggl( \frac{a}{1} \biggr)\biggr)
    .
  \end{equation*}
  This way it is guaranteed that $\Phi(\hat a) \neq 0$.
  From $\Lambda\bigl( \frac{\kappa(x)}{1} \bigr) \in \KK_\fin \otimes \WeylNil_r$ for all $x\in\lie g$ it follows that also
  $\Phi\bigl(\kappa(x)\bigr) \in \widehat\varphi\bigl( \KK_\fin \otimes \WeylNil_r \bigr) \subseteq \WeylNil_r$.

  It only remains to check that $\Phi$ is surjective:
  For all $j\in \{1,\dots,r\}$, the identities
  \begin{equation*}
    \Lambda\biggl(\frac{p'_j}{1}\biggr) = \frac{z}{1}\,\Lambda\biggl(\frac{p'_j}{z}\biggr) = \frac{z}{1} \otimes p_j
    \quad\quad\text{and}\quad\quad
    \Lambda\biggl(\frac{q'_j}{1}\biggr) = \frac{z}{1}\,\Lambda\biggl(\frac{q'_j}{z}\biggr) = \frac{z}{1} \otimes q_j
  \end{equation*}
  hold by $\KK$\=/linearity of $\Lambda$, with $p_1', \dots, p_r', q_1', \dots, q_r' \in \P$ from \eqref{eq:internal:pqPrimeDef}.
  Note that $\frac{z}{1},\frac{1}{z} \in \KK \cap \P_{z\mloc} = \KK_\fin$,
  and $\varphi\bigl(\frac{z}{1}\bigr)\, \varphi\bigl(\frac{1}{z}\bigr) = \varphi\bigl(\frac{1}{1}\bigr) = 1$.
  Therefore $\varphi\bigl(\frac{z}{1}\bigr) \in \CC \setminus \{0\}$, and
  \begin{equation*}
    \Phi(p'_j) = \widehat\varphi\biggl(\Lambda\biggl( \frac{p'_j}{1} \biggr) \biggr) = \varphi\biggl( \frac{z}{1} \biggr)\,p_j
    \quad\quad\text{and}\quad\quad
    \Phi(q'_j) = \widehat\varphi\biggl(\Lambda\biggl( \frac{q'_j}{1} \biggr) \biggr) = \varphi\biggl( \frac{z}{1} \biggr)\,q_j
    .
  \end{equation*}
  It follows that the image of $\Phi$ contains all $p_j,q_j$ for $j\in\{1,\dots,r\}$, and therefore equals whole $\Weyl{r}$.
\end{proof}

The proof of our main Theorem~\ref{theorem:main} is a straightforward combination of the above Lemma~\ref{lemma:mainLemma}
with some other previous results:

\begin{proof}[of Theorem~\ref{theorem:main}]
  Let $\lie g$ be a nilpotent finite-dimensional Lie $^*$\=/algebra, $(\A,\iota)$ an enveloping $^*$\=/algebra of $\lie g$,
  and $I$ a real ideal of $\A$.
  We have to show that there is a set $S$ of irreducible integrable $^*$\=/representations of
  $(\A,\iota)$ such that $I = \bigcap_{(\Dom,\pi) \in S} \ker(\pi)$.
  If $I=\A$, then we can choose $S\coloneqq \emptyset$.
  So consider the other case, $I\neq \A$, and let $\hat a \in \A \setminus I$.
  The proof is complete once we have constructed an irreducible integrable $^*$\=/representation
  $( \Dom, \pi )$ of $(\A,\iota)$ such that $I \subseteq \ker(\pi)$ and $\hat a \notin \ker(\pi)$.

  By Corollary~\ref{corollary:realIdealDecomposition}, $I$ is the intersection of finitely many real prime ideals of $\A$.
  In particular, there is a real prime ideal $J$ of $\A$ such that $I \subseteq J$ and $\hat a \notin J$.
  As $J$ automatically is a $^*$\=/ideal, the quotient $\P \coloneqq \A / J$ is a prime $^*$\=/algebra.
  We write $[\argument] \colon \A \to \P$ for the canonical projection onto the quotient (which is a $^*$\=/algebra morphism).
  Clearly $[\hat a] \neq [0]$ because $\hat a \notin J$.

  Define $\kappa \colon \lie g \to \P$, $x \mapsto \kappa(x) \coloneqq [\iota(x)]$.
  Then $\kappa$ is a Lie $^*$\=/algebra morphism and $\P$ is generated as a $^*$\=/algebra by $\kappa(\lie g) = [\iota(\lie g)]$,
  i.e., $(\P,\kappa)$ is an enveloping $^*$\=/algebra of $\lie g$.
  By Lemma~\ref{lemma:mainLemma}, there are $r\in \NN_0$ and a surjective $^*$\=/algebra morphism
  $\Phi \colon \P \to \Weyl{r}$ such that $[\hat a] \notin \ker(\Phi)$
  and $\Phi\bigl( \kappa(\lie g)\bigr) \subseteq \WeylNil_r$.

  Let $\bigl( \Schwartz(\RR^r), \pi_\cat\bigr)$ be the Schrödinger $^*$\=/representation of $\Weyl{r}$ and define
  the $^*$\=/algebra morphism $\pi_\Phi \colon \A \to \Adbar\bigl(\Schwartz(\RR^r)\bigr)$,
  \begin{equation*}
    a \mapsto \pi_\Phi(a) \coloneqq \pi_\cat\bigl( \Phi([a]) \bigr),
  \end{equation*}
  then $\bigl( \Schwartz(\RR^r), \pi_\Phi\bigr)$ is a $^*$\=/representation of $\A$
  and $I \subseteq \ker(\pi_\Phi)$ because $I \subseteq J = \ker([\argument])$.
  As $[\hat a] \notin \ker(\Phi)$ and as $\pi_\cat \colon \Weyl{r} \to \Adbar\bigl(\Schwartz(\RR^r)\bigr)$
  is injective by part~\ref{item:Schroedinger:faithful} of Proposition~\ref{proposition:Schroedinger},
  $\hat a \notin \ker(\pi_\Phi)$.

  It only remains to show integrability and irreducibility of $\bigl( \Schwartz(\RR^r), \pi_\Phi\bigr)$.
  Let $\iota' \coloneqq \Phi \circ \kappa \colon \lie g \to \Weyl{r}$,
  then $\Weyl{r}$ is generated as a $^*$\=/algebra by $\iota'(\lie g)$ because $\P$ is generated by $\kappa(\lie g)$
  and because $\Phi \colon \P \to \Weyl{r}$ is surjective. So $(\Weyl{r},\iota')$ is also an enveloping $^*$\=/algebra of $\lie g$.
  Moreover, $\iota'(\lie g) = \Phi\bigl( \kappa(\lie g) \bigr) \subseteq \WeylNil_r$ by construction of $\Phi$,
  so $\bigl( \Schwartz(\RR^r), \pi_\cat\bigr)$ is an irreducible integrable $^*$\=/representation of $(\Weyl{r},\iota')$
  by Theorem~\ref{theorem:integrable}.
  Finally, note that $\pi_\cat \circ \iota' = \pi_\cat \circ \Phi \circ \kappa = \pi_\Phi \circ \iota$ because $\kappa = [\argument] \circ \iota$.
  Therefore, Proposition~\ref{proposition:irredintstability} applies to the $^*$\=/representations
  $\bigl( \Schwartz(\RR^r), \pi_\cat\bigr)$ and $\bigl( \Schwartz(\RR^r), \pi_\Phi\bigr)$
  of the enveloping $^*$\=/algebras $(\Weyl{r},\iota')$ and $( \A, \iota )$, respectively,
  and shows that $\bigl( \Schwartz(\RR^r), \pi_\Phi\bigr)$ is an irreducible integrable $^*$\=/representation of
  $( \A, \iota )$.
\end{proof}

\begin{remark}
  Let us revisit the above proof of Theorem~\ref{theorem:main}:
  For any $\hat a \in \A \setminus I$ we construct, by passing to a prime quotient and then applying $\Phi$,
  a surjective $^*$\=/algebra morphism $\Phi' \coloneqq \Phi \circ [\argument] \colon \A \to \Weyl{r}$
  such that $I \subseteq \ker(\Phi')$ and $\hat a \notin \ker(\Phi')$. This shows that $I$ is the intersection of
  kernels of surjective $^*$\=/algebra morphisms into some Weyl $^*$\=/algebra $\Weyl{r}$ (with $r\in \NN_0$ not fixed).
  As the Weyl $^*$\=/algebra is simple (see, e.g., Proposition~\ref{proposition:WeylTotallySimple}) and formally real
  (it admits a faithful $^*$\=/representation, e.g., the Schrödinger $^*$\=/representation),
  the kernel of any such surjective $^*$\=/algebra morphism $\A \to \Weyl{r}$ is a real maximal ideal of $\A$.
  This already shows that $I$ is the intersection of all real maximal ideals that contain $I$, and in particular,
  of all real primitive ideals that contain $I$.
  In the second step we compose with the Schrödinger $^*$\=/representation and show that the resulting
  $^*$\=/representation of $\A$ is integrable and irreducible by applying Theorem~\ref{theorem:integrable}.
  Alternatively, one could also apply classification results for representations of nilpotent Lie algebras
  to show that every real primitive ideal of $\A$ (and even every primitive $^*$\=/ideal of $\A$) is the kernel of an
  irreducible integrable $^*$\=/representation; this follows from
  \cite[Lemma~27 and Thm.~12]{dixmier:representationsIrreductiblesDesAlgebresDeLieNilpotentes}.
  However, the technical details involved are cumbersome.
\end{remark}

\section{The real universal enveloping algebra} \label{sec:realuniv}
In this last section we consider the \emph{real} universal enveloping algebra $\Univ(\lie g)$
of a \emph{real} Lie algebra $\lie g$, with the canonical Lie algebra morphism
$\iotaUniv{\lie g} \colon \lie g \to \Univ(\lie g)$.
Recall that $\Univ(\lie g)$ is a Hopf algebra, in particular, its \emph{antipode} or \emph{principal antiautomorphism}
$S \colon \Univ(\lie g) \to \Univ(\lie g)$
is the algebra antihomomorphism (linear map such that $S(1) = 1$ and $S(ab) = S(b)\,S(a)$ for all $a,b\in \Univ(\lie g)$)
that is uniquely determined by $S(\iotaUniv{\lie g}(x)) \coloneqq -\iotaUniv{\lie g}(x)$ for all $x\in \lie g$;
clearly $S\circ S = \id_{\Univ(\lie g)}$, i.e., $S$ is an involution.

For any real vector space $V$, its \emph{complexification} is the direct sum $V_\CC \coloneqq V \oplus \I V$.
We will treat $V$ as a real linear subspace of the complex vector space $V_\CC$ via the inclusion $V \ni v \mapsto v+\I {\,\cdot\,} 0 \in V_\CC$.
Every $\RR$\=/linear map $L \colon V \to W$ into a complex vector space $W$ has a unique $\CC$\=/linear extension
$L_\CC \colon V_\CC \to W$, namely, $L_\CC(v+\I v') \coloneqq L(v) + \I L(v')$ for $v,v'\in V$. Similarly, every $\RR$\=/bilinear
map $B \colon V\times V \to W$ also has an analogous unique $\CC$\=/bilinear extension $B_\CC \colon V_\CC \times V_\CC \to W$.
In particular, Lie brackets and algebra multiplications can be extended $\CC$\=/bilinearly to the complexifications.
The complexification comes with a natural \emph{conjugation map} $\cc{\argument} \colon V_\CC \to V_\CC$, defined
as $\cc{v+\I v'} \coloneqq v-\I v'$ for all $v,v'\in V$. Clearly, $\cc{\argument}$ is an antilinear involution of $V_\CC$.

On the complexification $\lie g_\CC$ of $\lie g$ we define a $^*$\=/involution as $(x+\I y)^* \coloneqq -(\cc{x+\I y}) = -x + \I y$
for $x,y\in \lie g$, then $\lie g_\CC$ is a Lie $^*$\=/algebra and contains $\lie g$ as a real Lie subalgebra.
Similarly, on $\Univ(\lie g)_\CC$ we define a $^*$\=/involution as
$(a + \I b)^* \coloneqq \cc{S(a) + \I S(b)} = S(a) - \I S(b)$ for $a,b\in \Univ(\lie g)$,
then $\Univ(\lie g)_\CC$ is a $^*$\=/algebra and contains $\Univ(\lie g)$ as a real subalgebra.
By $\CC$\=/linear extension of $\iotaUniv{\lie g} \colon \lie g \to \Univ(\lie g) \subseteq \Univ(\lie g)_\CC$ we obtain
the Lie $^*$\=/algebra morphism $\iotaUnivC{\lie g} \colon \lie g_\CC \to \Univ(\lie g)_\CC$.
Note that our definitions of the $^*$\=/involutions on $\lie g_\CC$ and $\Univ(\lie g)_\CC$ ensure that
indeed $\iotaUnivC{\lie g}\bigl((x+\I y)^*\bigr) = \iotaUnivC{\lie g}(x+\I y)^*$ for all $x,y\in \lie g$.
Clearly, $\Univ(\lie g)_\CC$ is generated as a $^*$\=/algebra by $\iotaUnivC{\lie g}(\lie g_\CC)$, so
$\bigl(\Univ(\lie g)_\CC, \iotaUnivC{\lie g}\bigr)$ is an enveloping $^*$\=/algebra of $\lie g_\CC$.

By a \emph{$^*$\=/representation of $\bigl(\Univ(\lie g),\iotaUniv{\lie g}\bigr)$}
we mean a tuple $(\Dom,\pi)$ of a pre-Hilbert space $\Dom$ and an algebra morphism $\pi \colon \Univ(\lie g) \to \Adbar(\Dom)$
such that $\skal{\phi}{\pi(u)\,\psi} = \skal{\pi\bigl(S(u)\bigr)\,\phi}{\psi}$ for all
$\phi,\psi \in \Dom$. This means that the $\CC$\=/linear extension $\pi_\CC \colon \Univ(\lie g)_\CC \to \Adbar(\Dom)$
is a $^*$\=/algebra morphism, i.e., $(\Dom,\pi_\CC)$ is a $^*$\=/representation of $\bigl(\Univ(\lie g)_\CC, \iotaUnivC{\lie g} \bigr)$.
Such a $^*$\=/representation $(\Dom,\pi)$ of $\bigl(\Univ(\lie g),\iotaUniv{\lie g}\bigr)$
is called \emph{integrable} (and \emph{irreducible}) if $(\Dom, \pi_\CC)$ is an integrable (and irreducible)
$^*$\=/representation of $\bigl(\Univ(\lie g)_\CC, \iotaUnivC{\lie g} \bigr)$.
Every $^*$\=/representation $(\Dom,\hat\pi)$ of $\bigl(\Univ(\lie g)_\CC, \iotaUnivC{\lie g} \bigr)$
restricts to a $^*$\=/representation $(\Dom,\hat\pi\at{\Univ(\lie g)})$ of
$\bigl(\Univ(\lie g), \iotaUniv{\lie g} \bigr)$, where $\hat\pi\at{\Univ(\lie g)}(a) \coloneqq \hat\pi(a)$ for all
$a \in \Univ(\lie g) \subseteq \Univ(\lie g)_\CC$.
In this case, the $\CC$\=/linear extension is $(\hat\pi\at{\Univ(\lie g)})_\CC = \hat \pi$,
so $(\Dom,\hat\pi\at{\Univ(\lie g)})$ is integrable (and irreducible) if and only if
$(\Dom,\hat\pi)$ is integrable (and irreducible).

Real ideals on the real associative algebra $\Univ(\lie g)$ with involution $S$ are defined analogously to the
$^*$\=/algebraic case:

\begin{definition} \label{definition:Sreal}
  Let $\lie g$ be a real Lie algebra and $I$ an ideal of $\Univ(\lie g)$.
  We say that $I$ is \emph{real} if the following property holds:
  whenever $k\in \NN$ and $a_1,\dots,a_k \in \Univ(\lie g)$ fulfill $\sum_{j=1}^k S(a_j)\,a_j \in I$,
  then $a_j \in I$ for all $j\in\{1,\dots,k\}$.
\end{definition}

Transfering Theorem~\ref{theorem:main} to real universal enveloping algebras is straightforward.
The key ingredient in the proof are the following parts \ref{item:updown:wedge}, \ref{item:updown:wedgevee}, and \ref{item:updown:vee}:

\begin{lemma} \label{lemma:updown}
  Let $\lie g$ be a real Lie algebra. We introduce the following shorthand notation:
  \begin{itemize}
    \item If $I$ is an ideal of $\Univ(\lie g)$, then we write $I^{\wedge} \coloneqq I + \I I \subseteq \Univ(\lie g)_\CC$.
    \item If $J$ is an ideal of $\Univ(\lie g)_\CC$, then we write $J^{\vee} \coloneqq J \cap \Univ( \lie g) \subseteq \Univ( \lie g)$.
  \end{itemize}
  These two constructions have the following properties:
  \begin{enumerate}
    \item \label{item:updown:wedge}
      If $I$ is a real ideal of $\Univ( \lie g)$, then $I^{\wedge}$ is a real ideal of $\Univ(\lie g)_\CC$.
    \item \label{item:updown:wedgevee}
      For all ideals $I$ of $\Univ( \lie g)$, the identity $(I^\wedge)^\vee = I$ holds.
    \item \label{item:updown:vee}
      If $J$ is a real ideal of $\Univ(\lie g)_\CC$, then $J^{\vee}$ is a real ideal of $\Univ( \lie g)$.
    \item \label{item:updown:veewedge}
      For all ideals $J$ of $\Univ(\lie g)_\CC$, the inclusion $(J^\vee)^\wedge \subseteq J$ holds.
      Moreover, $(J^\vee)^\wedge = J$ if and only if $J$ is stable under conjugation,
      i.e., $J = \cc J$ with $\cc J \coloneqq \set{\cc a}{a\in J}$.
  \end{enumerate}
\end{lemma}
\begin{proof}
  Part~\ref{item:updown:wedge}:
  Let $I$ be a real ideal of $\Univ( \lie g)$, then it is clear that $I^{\wedge}$ is an ideal of $\Univ(\lie g)_\CC$.
  Now we check that $I^{\wedge}$ is real, so let $k\in \NN$ and $u_1, \dots, u_k \in \Univ(\lie g)_\CC$ such that
  $\sum_{j=1}^k u_j^* u_j \in I^{\wedge}$. Then there are uniquely determined $\hat v, \hat w \in \Univ(\lie g)$
  such that $\sum_{j=1}^k u_j^* u_j = \hat v + \I \hat w$, and from $\hat v + \I \hat w \in I^{\wedge} = I + \I I$
  it follows that $\hat v, \hat w \in I$.
  Similarly, for all $j\in\{1,\dots,k\}$ there are $v_j, w_j \in \Univ (\lie g)$ such that $u_j = v_j + \I w_j$. Then
	\begin{equation*}
    \sum_{j=1}^k u_j^* u_j
    =
    \sum_{j=1}^k \bigl( S(v_j) - \I S(w_j) \bigr) (v_j + \I w_j)
    =
    \sum_{j=1}^k \bigl( S(v_j) v_j + S(w_j) w_j \bigr) + \I \sum_{j=1}^k \bigl( S(v_j) w_j - S(w_j) v_j \bigr)
    ,
	\end{equation*}
	and therefore $\sum_{j=1}^k \bigl( S(v_j) v_j + S(w_j) w_j \bigr) = \hat v \in I$.
	As $I$ is real, this implies $v_1,\dots,v_k,w_1,\dots,w_k \in I$, which means that $u_1,\dots,u_k \in I+\I I = I^{\wedge}$.

	Part~\ref{item:updown:wedgevee}:
	Let $I$ be any ideal of $\Univ( \lie g)$, then
	$(I^{\wedge})^{\vee} = \set[\big]{u \in \Univ( \lie g)}{ u + \I{\,\cdot\,}0 \in I + \I I } = I$.

	Part~\ref{item:updown:vee}:
	Let $J$ be a real ideal of $\Univ(\lie g)_\CC$, then it is clear that $J^{\vee}$ is an ideal.
	In order to check that $J^{\vee}$ is real, let
	$k\in \NN$ and $u_1, \dots, u_k \in \Univ(\lie g)$ such that $\sum_{j=1}^k S(u_j) u_j \in J^{\vee}$.
	Note that the $^*$\=/involution of $\Univ(\lie g)_\CC$ and the involution $S$ of $\Univ(\lie g)$
	coincide on the real linear subspace $\Univ(\lie g)$ of $\Univ(\lie g)_\CC$,
	so $\sum_{j=1}^k u_j^* u_j = \sum_{j=1}^k S(u_j) u_j \in J^{\vee} \subseteq J$,
	which implies $u_1, \dots, u_k \in J$ since $J$ is real.
	It follows that $u_1, \dots, u_k \in J \cap \Univ(\lie g) = J^{\vee}$.

	Part~\ref{item:updown:veewedge}:
	Let $J$ be any ideal of $\Univ(\lie g)_\CC$, then $J^{\vee} \subseteq J$ and $\I J^{\vee} \subseteq J$,
	and consequently $(J^{\vee})^{\wedge} = J^{\vee} + \I J^{\vee} \subseteq J$.
	Now assume that $J = (J^{\vee})^{\wedge}$. This means $J = J^{\vee} + \I J^{\vee}$, so $\cc J = J^{\vee} - \I J^{\vee} =
	J^{\vee} + \I J^{\vee} = J$.
	Conversely, assume that $J = \cc J$. It remains to show that in this case $J \subseteq (J^{\vee})^{\wedge}$,
	so let $u \in J$. Then there are uniquely determined $v,w \in \Univ(\lie g)$ such that $u = v+\I w$,
	and the identities $\frac 1 2 (u + \overline u) = v$ and $\frac 1 {2 \I} (u- \overline u) = w$ show that also $v,w \in J$.
	So $v,w \in J\cap \Univ(\lie g) = J^{\vee}$ and $u \in J^{\vee} + \I J^{\vee} = (J^{\vee})^{\wedge}$.
\end{proof}

\begin{theorem} \label{theorem:realUniv}
  Let $\lie g$ be a nilpotent finite-dimensional real Lie algebra and $I$ an ideal of $\Univ(\lie g)$.
  Then the following are equivalent:
  \begin{enumerate}
    \item \label{item:realUniv:real} $I$ is real,
    \item \label{item:realUniv:ker} there is a set $X$ of irreducible integrable $^*$\=/representations of
      $\bigl(\Univ(\lie g), \iotaUniv{\lie g}\bigr)$ such that
      \begin{equation}
        \label{eq:theorem:realUniv}
        I = \bigcap_{(\Dom,\pi)\in X} \ker(\pi),
      \end{equation}
      where, by convention, the empty intersection is $\bigcap_{(\Dom,\pi)\in \emptyset} \ker(\pi) \coloneqq \Univ(\lie g)$.
  \end{enumerate}
\end{theorem}
\begin{proof}
  First assume that $I$ is real. By part~\ref{item:updown:wedge} of Lemma~\ref{lemma:updown},
  $I^{\wedge}$ is a real ideal of $\Univ(\lie g)_\CC$. Recall that $\bigl(\Univ(\lie g)_\CC, \iotaUnivC{\lie g}\bigr)$
  is an enveloping $^*$\=/algebra of $\lie g_\CC$ and that the real Nullstellensatz holds for $\lie g_\CC$ by Theorem~\ref{theorem:main}
  (clearly, if $\lie g$ is nilpotent, then $\lie g_\CC$ is also nilpotent).
  Therefore, there is a set $Y$ of irreducible integrable $^*$\=/representations of $\bigl(\Univ(\lie g)_\CC, \iotaUnivC{\lie g}\bigr)$
  such that $I^{\wedge} = \bigcap_{(\Dom,\pi)\in Y} \ker(\pi)$.

  Consider any $(\Dom,\pi) \in Y$. By restriction to $\Univ(\lie g) \subseteq \Univ(\lie g)_\CC$
  we obtain an irreducible integrable $^*$\=/representation $\bigl(\Dom,\pi\at{\Univ(\lie g)}\bigr)$
  of $\bigl(\Univ(\lie g), \iotaUniv{\lie g} \bigr)$, and $\ker (\pi\at{\Univ(\lie g)}) = \ker(\pi) \cap \Univ(\lie g)$.
  By part~\ref{item:updown:wedgevee} of Lemma~\ref{lemma:updown},
  \begin{equation*}
    I
    =
    (I^{\wedge})^{\vee}
    =
    \biggl( \bigcap_{(\Dom,\pi)\in Y} \ker(\pi) \biggr) \cap \Univ(\lie g)
    =
    \bigcap_{(\Dom,\pi)\in Y} \bigl( \ker(\pi) \cap \Univ(\lie g) \bigr)
    =
    \bigcap_{(\Dom,\pi)\in Y} \ker \bigl(\pi\at{\Univ(\lie g)}\bigr)
    .
  \end{equation*}

  So set $X \coloneqq \set{(\Dom,\pi\at{\Univ(\lie g)})}{(\Dom,\pi) \in Y}$,
  then $X$ is a set of irreducible integrable $^*$\=/representations of $\bigl( \Univ(\lie g), \iotaUniv{\lie g}\bigr)$
  such that \eqref{eq:theorem:realUniv} holds. This proves the implication \ref{item:realUniv:real}~$\implies$~\ref{item:realUniv:ker}.

  Conversely, assume there is a set $X$ of irreducible integrable $^*$\=/representations of
  $\bigl(\Univ(\lie g), \iotaUniv{\lie g}\bigr)$ such that \eqref{eq:theorem:realUniv} holds.
  For every $^*$\=/representation $(\Dom,\pi)$ of $\bigl(\Univ(\lie g), \iotaUniv{\lie g}\bigr)$,
  the kernel $\ker(\pi)$ is real because $\ker(\pi) = \ker(\pi_\CC) \cap \Univ(\lie g)$
  and because the restriction $\ker(\pi_\CC) \cap \Univ(\lie g) = \ker(\pi_\CC)^{\vee}$ of the real ideal $\ker(\pi_\CC)$ is real
  by part~\ref{item:updown:vee} of Lemma~\ref{lemma:updown}.
  It is easy to check that an arbitrary intersection of real ideals of $\Univ(\lie g)$ is real,
  so \eqref{eq:theorem:realUniv} implies that $I$ is real.
  This proves the implication \ref{item:realUniv:ker}~$\implies$~\ref{item:realUniv:real}.
\end{proof}

Finally, we note that one cannot deduce Theorem~\ref{theorem:main} from Theorem~\ref{theorem:realUniv}
by the above argument, mutatis mutandis, because in part~\ref{item:updown:veewedge} of Lemma~\ref{lemma:updown},
equality holds only for ideals that are stable under conjugation.
This effect can already be observed in the $1$-dimensional case:

\begin{example}
	Let $\lie g \coloneqq \RR$ be the abelian $1$-dimensional Lie algebra, generated by one element $x$.
	Then $\Univ(\lie g) \cong \RR[x]$, and the involution $S$ of $\Univ(\lie g)$ corresponds to the involutive
	(anti-)automorphism of $\RR[x]$ that maps $x$ to $-x$.
	As $\RR$ is abelian, the irreducible integrable $^*$\=/representations of $\Univ(\lie g)$ are $1$\=/dimensional
	(see Section~\ref{sec:integrability}).	Explicitly,
	they are given by evaluation at $\I \lambda$ for some $\lambda \in \RR$, i.e., by $\delta_{\I \lambda} \colon \RR[x] \to \CC$,
	\begin{equation*}
    p \mapsto \delta_{\I \lambda}(p) \coloneqq p(\I \lambda)
    .
	\end{equation*}
  Note that $\cc{p(\I \lambda)} = p(-\I \lambda)$ for all $p \in \RR[x]$,
  and therefore $\ker(\delta_{\I \lambda}) = \ker(\delta_{-\I \lambda})$.
  On the complexifications, however, $\ker\bigl( (\delta_{\I \lambda})_\CC\bigr) \neq \ker\bigl( (\delta_{-\I \lambda})_\CC \bigr)$
  for $\lambda \neq 0$,
  e.g., $(\delta_{\I \lambda})_\CC(x-\I \lambda) = 0$ while $(\delta_{-\I \lambda})_\CC(x-\I \lambda) = -2\I \lambda$.
  This demonstrates how $\RR[x]$ and $\Univ(\lie g)$ have less (real) ideals than their complexifications,
  as indicated by part~\ref{item:updown:veewedge} of Lemma~\ref{lemma:updown}. Therefore,
  the real Nullstellensatz Theorem~\ref{theorem:main} for $\Univ(\lie g)_\CC$ is stronger, i.e., it
  applies to more ideals than the real Nullstellensatz Theorem~\ref{theorem:realUniv} for $\Univ(\lie g)$.
\end{example}

\end{onehalfspace}


\end{document}